\documentclass[11pt,a4paper,leqno]{amsart}

\title
[Propagation of anisotropic Gelfand--Shilov wave front sets]
{Propagation of anisotropic Gelfand--Shilov wave front sets}

\author[P. Wahlberg]{Patrik Wahlberg}

\address{Dipartimento di Scienze Matematiche, Politecnico di Torino, Corso Duca degli Abruzzi 24,
10129 Torino, Italy}

\address{ORCID: 0000-0003-4740-0629}

\email{patrik.wahlberg[AT]polito.it}

\usepackage{latexsym}
\usepackage[a4paper]{geometry}
\usepackage{amsmath}
\usepackage{amssymb}
\usepackage{amsthm}
\usepackage{bbm}
\usepackage{amsfonts}
\usepackage{mathrsfs}
\usepackage{calc}
\usepackage{cite}
\usepackage{color}
\usepackage{graphicx}
\usepackage{enumerate}

\numberwithin{equation}{section}          

\newtheorem{thm}{Theorem}
\numberwithin{thm}{section}

\newcommand{\rubrik}{}
\newtheorem{prop}[thm]{Proposition}

\newtheorem{lem}[thm]{Lemma}

\theoremstyle{definition}

\newtheorem{defn}[thm]{Definition}

\theoremstyle{remark}

\newtheorem{rem}[thm]{Remark}              

\newcommand{\scal}[2]{\langle #1,#2\rangle}

\newcommand{\ro}{\mathbf R}
\newcommand{\no}{\mathbf N}
\newcommand{\rr}[1]{\mathbf R^{#1}}
\newcommand{\sr}[1]{\mathbf S^{#1}}
\newcommand{\sro}[1]{\mathbf S}
\newcommand{\nn}[1]{\mathbf N^{#1}}

\newcommand{\dd}{\mathrm {d}}

\newcommand{\nm}[2]{\Vert #1\Vert _{#2}}

\newcommand{\ep}{\varepsilon}
\newcommand{\fy}{\varphi}

\newcommand{\cdo}{\, \cdot \, }

\newcommand{\supp}{\operatorname{supp}}

\newcommand{\eabs}[1]{\langle #1\rangle}

\newcommand{\GL}{\operatorname{GL}}

\newcommand{\rB}{\operatorname{B}}

\newcommand{\WF}{\mathrm{WF}}

\newcommand{\cS}{\mathscr{S}}

\newcommand{\cF}{\mathscr{F}}

\newcommand{\cK}{\mathscr{K}}

\newcommand{\J}{\mathcal{J}}
\newcommand{\wt}{\widetilde}
\newcommand{\wh}{\widehat}

\def\la{\langle}
\def\ra{\rangle}
\newcommand{\leqs}{\leqslant}
\newcommand{\geqs}{\geqslant}

\begin{document}

\begin{abstract}
We show a result on propagation of the anisotropic Gelfand--Shilov wave front set
for linear operators with Schwartz kernel which is a Gelfand--Shilov ultradistribution of Beurling type. 
This anisotropic wave front set is parametrized by two positive parameters relating the space and frequency variables. 
The anisotropic Gelfand--Shilov wave front set of the Schwartz kernel of the operator is assumed to satisfy a graph type criterion. 
The result is applied to a class of evolution equations that generalizes the Schr\"odinger equation 
for the free particle. 
The Laplacian is replaced by a partial differential operator 
defined by a symbol which is a polynomial with real coefficients and order at least two. 
\end{abstract}


\maketitle

\section{Introduction}\label{sec:intro}

The paper treats the anisotropic Gelfand--Shilov wave front set and its
propagation for a class of continuous linear operators. 

The Gabor wave front set, introduced by H\"ormander in 1991 \cite{Hormander1}, is a closed conic subset of the phase space $T^* \rr d \setminus 0$
that consists of globally singular directions of tempered distributions. 
More precisely it records directions in $T^* \rr d \setminus 0$ in a conical neighborhood of which the short-time Fourier transform of a tempered 
distribution does not decay super-polynomially. 
It is empty precisely when the tempered distribution is a Schwartz function, and thus it records local smoothness as well as rapid decay at infinity comprehensively. 
These singularities thus merits the term global. 

Several recent works \cite{Cordero1,PRW1,Rodino2,Schulz1,Wahlberg1} concern the Gabor wave front set and generalizations.
In particular it has been shown to coincide with Nakamura's homogeneous wave front set \cite{Nakamura1,Schulz1}. 
Concerning propagation of singularities already the original paper \cite{Hormander1} treated the action of a linear continuous operator on the Gabor wave front set. 
In \cite{PRW1,Wahlberg1} propagation of the Gabor wave front set for the solution operator to an evolution equation with quadratic Hamiltonian is studied. Then the singular space, introduced by Hitrik and Pravda--Starov \cite{Hitrik1}, plays a major role. 

In \cite{Carypis1} the Gabor wave front set is adapted to the functional framework of equal index Gelfand--Shilov spaces of Beurling type 
and their dual ultradistribution spaces. This means that the super-polynomial decay for the Gabor wave front set is replaced by super-exponential decay with a subgaussian power parameter $\frac1s < 2$. A study of propagation of this $s$-Gelfand--Shilov wave front set for evolution equations of quadratic type is also contained in \cite{Carypis1}. 

In \cite{Rodino3} the isotropic $s$-Gelfand--Shilov wave front set is generalized into an anisotropic Gelfand--Shilov wave front set 
parametrized by two parameters $t,s > 0$ such that $t + s > 1$.
The parameters relate the space and frequency variables.  
The anisotropic Gelfand--Shilov wave front set is defined for Gelfand--Shilov ultradistributions of Beurling type 
with decay index $t$ and regularity index $s$. 
The super-exponential decay along straight lines in phase space $T^* \rr d \setminus 0$ used for the isotropic Gelfand--Shilov wave front set is then replaced by 
super-exponential decay along curves of the form 
\begin{equation*}
\ro_+ \ni \lambda \mapsto ( \lambda^t x, \lambda^s \xi) \in T^* \rr d \setminus 0
\end{equation*}
where $(x,\xi) \in T^* \rr d \setminus 0$. 
We call the resulting wave front set the anisotropic $t,s$-Gelfand--Shilov wave front set. 
It is denoted $\WF^{t,s} (u) \subseteq T^* \rr d \setminus 0$ for a Gelfand--Shilov ultradistribution $u \in \left( \Sigma_t^s \right)'(\rr d)$. 
If $t = s$ we recapture the $s$-Gelfand--Shilov wave front set. 
In \cite{Rodino3} microlocal analysis for the anisotropic $t,s$-Gelfand--Shilov wave front set is developed. 
In particular a result on microlocality for pseudodifferential operators in the anisotropic framework is shown, 
with a symbol class taken from \cite{Abdeljawad1}.
These operators are continuous on the Gelfand--Shilov space $\Sigma_t^s (\rr d)$ and 
extends to continuous operators on $\left( \Sigma_t^s \right)'(\rr d)$. 


The following main result in this paper concerns propagation of the anisotropic $t,s$-Gelfand--Shilov wave front set
for a continuous linear operator $\cK: \Sigma_t^s (\rr d) \to\left( \Sigma_t^s \right)'(\rr d)$ defined by a Schwartz kernel $K \in \left( \Sigma_t^s \right)' (\rr {2d})$. 

Suppose that the $t,s$-Gelfand--Shilov wave front set of $K$ contains no point of the form $(x, 0, \xi, 0)  \in T^* \rr {2d} \setminus 0$
nor of the form $(0, y, 0, -\eta) \in T^* \rr {2d} \setminus 0$. 
(Loosely speaking this means that $\WF^{t,s} (K)$ resembles the graph of an invertible matrix.)
Then $\cK: \Sigma_t^s(\rr d) \to \Sigma_t^s(\rr d)$ is continuous, extends uniquely to a continuous linear operator $\cK:  \left( \Sigma_t^s \right)' (\rr d) \to  \left( \Sigma_t^s \right)' (\rr d)$, 
and for $u \in  \left( \Sigma_t^s \right)' (\rr d)$ we have
\begin{equation}\label{eq:linearpropagation}
\WF^{t,s} (\cK u) \subseteq \WF^{t,s} (K)' \circ \WF^{t,s} (u)  
\end{equation}
where
\begin{equation*}
A' \circ B  = \{ (x,\xi) \in \rr {2d}: \,  \exists (y,\eta) \in B: \, (x,y,\xi,-\eta) \in A \}
\end{equation*}
for $A \subseteq T^* \rr {2d}$ and $B \subseteq T^* \rr d$. 

The inclusion \eqref{eq:linearpropagation} is conceptually similar to propagation results for other 
types of wave front sets, local \cite{Hormander0}, or global \cite{Carypis1,PRW1,Wahlberg1}. 

As an application of the inclusion \eqref{eq:linearpropagation} we study propagation of the anisotropic $t,s$-Gelfand wave front set
for the initial value Cauchy problem for an evolution equation of the form 
\begin{equation*}
\left\{
\begin{array}{rl}
\partial_t u(t,x) + i p(D_x) u (t,x) & = 0, \quad x \in \rr d, \\
u(0,\cdot) & = u_0 
\end{array}
\right.
\end{equation*}
where $p: \rr d \to \ro$ is a polynomial with real coefficients of order $m \geqs 2$. 
This generalizes the Schr\"odinger equation for the free particle where $m = 2$ and $p(\xi) = |\xi|^2$. 

Provided $s > \frac{1}{m-1}$ we show that $\WF^{s(m-1),s}$ of the solution $e^{- i t p(D_x)} u_0$ at time $t \in \ro$ equals $\WF^{s(m-1),s} (u_0)$ 
transported by the Hamilton flow $\chi_t $ with respect to the principal part $p_m $ of $p$, that is 
\begin{equation*}
( x (t),\xi (t) ) = \chi_t (x_0, \xi_0)
= (x_0 + t \nabla p_m (\xi_0), \xi_0), \quad t \in \ro, \quad (x_0, \xi_0 ) \in T^* \rr d \setminus 0. 
\end{equation*}

This conclusion is again conceptually similar to other results on propagation of singularities 
\cite{Hormander0,Carypis1,Wahlberg1}, and generalizes known results when $p$ is a homogeneous 
quadratic form \cite{PRW1}. 

We also show that the propagator $e^{- i t p(D_x)}$ for any $t \in \ro$ is continuous on $\Sigma_{r}^s(\rr d)$ for any $r,s > 0$ such that $r \geqs s(m-1) > 1$,
using the criterion mentioned above on the $r,s$-Gelfand--Shilov wave front set of
the Schwartz kernel of the propagator. 
This technique to prove continuity on Gelfand--Shilov spaces avoids direct estimates for seminorms, and
we hope it may be useful in other contexts. 

Several ideas and techniques for our works on anisotropic global microlocal analysis are borrowed from the literature on anisotropic local microlocal analysis (see e.g. \cite{Parenti1}). 
In these works the anisotropy refers mostly to the dual (frequency) variables only, for fixed space variables, 
whereas our anisotropy refers to the space and frequency variables comprehensively. 

The article is organized as follows. 
Notations and definitions are collected in Section \ref{sec:prelim}. 
Section \ref{sec:seminorms} treats a family of seminorms for Gelfand--Shilov spaces defined 
using the short-time Fourier transform. 
Section \ref{sec:anisotropicGSWF} recalls the definition of 
the anisotropic $t,s$-Gelfand--Shilov wave front set and a result on tensorization is proved.
We devote Section \ref{sec:propagation} to a proof of the main result on propagation of the 
anisotropic $t,s$-Gelfand--Shilov wave front set for linear operators. 
In Section \ref{sec:chirp} we generalize \cite[Theorem~4.2~(i)]{Rodino3} and find an
inclusion for anisotropic Gelfand--Shilov wave front sets of multivariable chirp functions. 
These are exponentials with real polynomial phase functions.
Finally Section \ref{sec:schrodinger} treats an application of our propagation result to a class of evolution equations of Schr\"odinger type.

\section{Preliminaries}\label{sec:prelim}

The unit sphere in $\rr d$ is denoted $\sr {d-1} \subseteq \rr d$. 
A ball of radius $r > 0$ centered in $x \in \rr d$ is denoted $\rB_r (x)$, $\rB_r(0) = \rB_r$, 
and $e_j \in \rr d$ is the vector of zeros except for position $j$, $1 \leqs j \leqs d$, where it is one. 
The transpose of a matrix $A \in \rr {d \times d}$ is denoted $A^T$ and the inverse transpose 
of $A \in \GL(d,\ro)$ is $A^{-T}$. 
We write $f (x) \lesssim g (x)$ provided there exists $C>0$ such that $f (x) \leqs C \, g(x)$ for all $x$ in the domain of $f$ and of $g$. 
If $f (x) \lesssim g (x) \lesssim f(x)$ then we write $f \asymp g$. 
We use the bracket $\eabs{x} = (1 + |x|^2)^{\frac12}$ for $x \in \rr d$. 
Peetre's inequality with optimal constant \cite[Lemma~2.1]{Rodino3} is
\begin{equation*}
\eabs{x+y}^s \leqs \left( \frac{2}{\sqrt{3}} \right)^{|s|} \eabs{x}^s\eabs{y}^{|s|}\qquad x,y \in \rr d, \quad s \in \ro. 
\end{equation*}
The normalization of the Fourier transform is
\begin{equation*}
 \cF f (\xi )= \widehat f(\xi ) = (2\pi )^{-\frac d2} \int _{\rr
{d}} f(x)e^{-i\scal  x\xi }\, \dd x, \qquad \xi \in \rr d, 
\end{equation*}
for $f\in \cS(\rr d)$ (the Schwartz space), where $\scal \cdo \cdo$ denotes the scalar product on $\rr d$. 
The conjugate linear action of a (ultra-)distribution $u$ on a test function $\phi$ is written $(u,\phi)$, consistent with the $L^2$ inner product $(\cdo ,\cdo ) = (\cdo ,\cdo )_{L^2}$ which is conjugate linear in the second argument. 

Denote translation by $T_x f(y) = f( y-x )$ and modulation by $M_\xi f(y) = e^{i \scal y \xi} f(y)$ 
for $x,y,\xi \in \rr d$ where $f$ is a function or distribution defined on $\rr d$. 
The composed operator is denoted $\Pi(x,\xi) = M_\xi T_x$. 
Let $\fy \in \cS(\rr d) \setminus \{0\}$. 
The short-time Fourier transform (STFT) of a tempered distribution $u \in \cS'(\rr d)$ is defined by 
\begin{equation*}
V_\fy u (x,\xi) = (2\pi )^{-\frac d2} (u, M_\xi T_x \fy) = \cF (u T_x \overline \fy)(\xi), \quad x,\xi \in \rr d. 
\end{equation*}
Then $V_\fy u$ is smooth and polynomially bounded \cite[Theorem~11.2.3]{Grochenig1}, that is 
there exists $k \geqs 0$ such that 
\begin{equation}\label{eq:STFTtempered}
|V_\fy u (x,\xi)| \lesssim \eabs{(x,\xi)}^{k}, \quad (x,\xi) \in T^* \rr d.  
\end{equation}
We have $u \in \cS(\rr d)$ if and only if
\begin{equation}\label{eq:STFTschwartz}
|V_\fy u (x,\xi)| \lesssim \eabs{(x,\xi)}^{-N}, \quad (x,\xi) \in T^* \rr d, \quad \forall N \geqs 0.  
\end{equation}

The inverse transform is given by
\begin{equation}\label{eq:STFTinverse}
u = (2\pi )^{-\frac d2} \iint_{\rr {2d}} V_\fy u (x,\xi) M_\xi T_x \fy \, \dd x \, \dd \xi
\end{equation}
provided $\| \fy \|_{L^2} = 1$, with action under the integral understood, that is 
\begin{equation}\label{eq:moyal}
(u, f) = (V_\fy u, V_\fy f)_{L^2(\rr {2d})}
\end{equation}
for $u \in \cS'(\rr d)$ and $f \in \cS(\rr d)$, cf. \cite[Theorem~11.2.5]{Grochenig1}.

\subsection{Spaces of functions and ultradistributions}

In this paper we work with Beurling type Gelfand--Shilov spaces and their dual ultradistribution spaces \cite{Gelfand2}. 

Let $s,t, h > 0$. 
The space denoted $\mathcal S_{t,h}^s(\rr d)$
is the set of all $f\in C^\infty (\rr d)$ such that
\begin{equation}\label{eq:seminormSigmats1}
\nm f{\mathcal S_{t,h}^s}\equiv \sup \frac {|x^\alpha D ^\beta
f(x)|}{h^{|\alpha + \beta |} \alpha !^t \, \beta !^s}
\end{equation}
is finite, where the supremum is taken over all $\alpha ,\beta \in
\mathbf N^d$ and $x\in \rr d$.
The function space $\mathcal S_{t,h}^s$ is a Banach space which increases
with $h$, $s$ and $t$, and $\mathcal S_{t,h}^s \subseteq \cS$.
The topological dual $(\mathcal S_{t,h}^s)'(\rr d)$ is
a Banach space such that $\cS'(\rr d) \subseteq (\mathcal S_{t,h}^s)'(\rr d)$.

The Beurling type \emph{Gelfand--Shilov space}
$\Sigma _t^s(\rr d)$ is the projective limit 
of $\mathcal S_{t,h}^s(\rr d)$ with respect to $h$ \cite{Gelfand2}. This means
\begin{equation}\label{GSspacecond1}
\Sigma _t^s(\rr d) = \bigcap _{h>0} \mathcal S_{t,h}^s(\rr d)
\end{equation}
and the Fr{\'e}chet space topology of $\Sigma _t^s (\rr d)$ is defined by the seminorms $\nm \cdot{\mathcal S_{t,h}^s}$ for $h>0$.
 
We have $\Sigma _t^s(\rr d)\neq \{ 0\}$ if and only if $s + t > 1$ \cite{Petersson1}. 
The topological dual of $\Sigma _t^s(\rr d)$ is the space of (Beurling type) \emph{Gelfand--Shilov ultradistributions} \cite[Section~I.4.3]{Gelfand2}
\begin{equation}\tag*{(\ref{GSspacecond1})$'$}
(\Sigma _t^s)'(\rr d) =\bigcup _{h>0} (\mathcal S_{t,h}^s)'(\rr d).
\end{equation}

The dual space $(\Sigma _t^s)'(\rr d)$ may be equipped with several topologies: the weak$^*$ topology, the strong topology, the Mackey topology, and the topology defined by the union \eqref{GSspacecond1}$'$ as an inductive limit topology \cite{Schaefer1}. The latter topology is the strongest such that the inclusion $(\mathcal S_{t,h}^s)'(\rr d) \subseteq (\Sigma _t^s)'(\rr d)$ is continuous for all $h > 0$.  
We use the weak$^*$ topology on $(\Sigma _t^s)'(\rr d)$ in this paper. 

The Roumieu type Gelfand--Shilov space is the union 
\begin{equation*}
\mathcal S_t^s(\rr d) = \bigcup _{h>0}\mathcal S_{t,h}^s(\rr d)
\end{equation*}
equipped with the inductive limit topology \cite{Schaefer1}, that is 
the strongest topology such that each inclusion $\mathcal S_{t,h}^s(\rr d) \subseteq\mathcal S_t^s(\rr d)$
is continuous. 
Then $\mathcal S _t^s(\rr d)\neq \{ 0\}$ if and only if $s+t \geqs 1$ \cite{Gelfand2}. 
The corresponding (Roumieu type) Gelfand--Shilov ultradistribution space is 
\begin{equation*}
(\mathcal S_t^s)'(\rr d) = \bigcap _{h>0} (\mathcal S_{t,h}^s)'(\rr d). 
\end{equation*}
For every $s,t > 0$ such that $s+t > 1$, and for any $\ep > 0$ we have
\begin{equation*}
\Sigma _t^s (\rr d)\subseteq \mathcal S_t^s(\rr d)\subseteq
\Sigma _{t+\ep}^{s+\ep}(\rr d).
\end{equation*}
We will not use the Roumieu type spaces in this article but mention them as a service to a reader interested in a wider context. 

We write $\Sigma _s^s (\rr d) = \Sigma_s (\rr d)$ and $(\Sigma _s^s)' (\rr d) = \Sigma_s' (\rr d)$. 
Then $\Sigma_s(\rr d) \neq \{ 0 \}$ if and only if $s > \frac12$.

The Gelfand--Shilov (ultradistribution) spaces enjoy invariance properties, with respect to 
translation, dilation, tensorization, coordinate transformation and (partial) Fourier transformation.
The Fourier transform extends 
uniquely to homeomorphisms on $\mathscr S'(\rr d)$, from $(\mathcal
S_t^s)'(\rr d)$ to $(\mathcal
S_s^t)'(\rr d)$, and from $(\Sigma _t^s)'(\rr d)$ to $(\Sigma _s^t)'(\rr d)$, and restricts to 
homeomorphisms on $\mathscr S(\rr d)$, from $\mathcal S_t^s(\rr d)$ to $\mathcal S_s^t(\rr d)$, 
and from $\Sigma _t^s(\rr d)$ to $\Sigma _s^t(\rr d)$, and to a unitary operator on $L^2(\rr d)$.

Let $u \in (\Sigma_t^s)' (\rr d)$ with $s + t > 1$. 
If $\psi \in \Sigma_t^s (\rr d) \setminus 0$ then 
\begin{equation}\label{eq:STFTGFstdistr}
| V_\psi u (x,\xi)| \lesssim e^{r (|x|^{\frac1t} + |\xi|^{\frac1s})}
\end{equation}
for some $r > 0$, and $u \in \Sigma_t^s (\rr d)$ if and only if 
\begin{equation}\label{eq:STFTGFstfunc}
| V_\psi u (x,\xi)| \lesssim e^{-r (|x|^{\frac1t} + |\xi|^{\frac1s})}
\end{equation}
for all $r > 0$. See e.g. \cite[Theorems~2.4 and 2.5]{Toft1}.
If $u \in (\Sigma_t^s)'(\rr d)$, $f \in \Sigma_t^s(\rr d)$, $\fy \in \Sigma_t^s(\rr d)$ and $\| \fy \|_{L^2} = 1$
then \eqref{eq:moyal} holds true.

Working with Gelfand--Shilov spaces we will often use the inequality (cf. \cite{Abdeljawad1})
\begin{equation*}
|x+y|^{\frac1s} \leqs \kappa(s^{-1} ) ( |x|^{\frac1s} + |y|^{\frac1s}), \quad x,y \in \rr d, \quad s > 0, 
\end{equation*}
where 
\begin{equation*}
\kappa (t)
= 
\left\{
\begin{array}{ll}
1 & \mbox{if} \quad 0 < t \leqs 1 \\
2^{t-1} & \mbox{if} \quad t > 1
\end{array}
\right. ,
\end{equation*}
which implies 
\begin{equation}\label{eq:exppeetre}
\begin{aligned}
e^{r |x+y|^{\frac1s} } & \leqslant e^{ \kappa(s^{-1} ) r |x|^{\frac1s}} e^{ \kappa(s^{-1} )r |y|^{\frac1s}},  \quad x,y \in \rr d, \quad r >0, \\
e^{- r \kappa(s^{-1} ) |x+y|^{\frac1s} } & \leqslant e^{- r |x|^{\frac1s}} e^{ \kappa(s^{-1} ) r |y|^{\frac1s}}, \quad x,y \in \rr d, \quad r >0. 
\end{aligned}
\end{equation}

We will use the following estimate based on $|\alpha|! \leqs \alpha! d^{|\alpha|}$ for $\alpha \in \nn d$ \cite[Eq.~(0.3.3)]{Nicola1}. For any $s > 0$, $h > 0$ and any $\alpha \in \nn d$ we have 
\begin{equation}\label{eq:expestimate0}
\alpha!^{-s} h^{- |\alpha|} 
= \left( \frac{h^{- \frac{|\alpha|}{s}}}{\alpha!} \right)^s
\leqs  \left( \frac{ \left( d h^{- \frac{1}{s}} \right)^{|\alpha|}}{|\alpha|!} \right)^s
\leqs e^{s d h^{- \frac{1}{s}}}. 
\end{equation}
%

\section{Seminorms on Beurling type Gelfand--Shilov spaces}\label{sec:seminorms}

We need the following result on seminorms in the space $\Sigma_t^s(\rr d)$ when $t + s > 1$. 
The result appears implicitly in the literature (cf. \cite[Theorem~2.4]{Toft1}) but we give a detailed proof as a service to the reader. 

\begin{lem}\label{lem:Sigmatsseminorm}
Let $t,s > 0$ satisfy $t + s > 1$, 
and let $\fy \in \Sigma_t^s(\rr d) \setminus 0$. 
The collection of seminorms
\begin{equation}\label{eq:seminormsSigmats2}
\Sigma_t^s(\rr d) \ni f \mapsto \sup_{(x,\xi) \in \rr {2d}} e^{r \left( |x|^{\frac1t} + |\xi|^{\frac1s} \right) } |V_\fy f (x,\xi)|, \quad r > 0,
\end{equation}
defines the same topology on $\Sigma_t^s(\rr d)$ as does the collection
of seminorms \eqref{eq:seminormSigmats1} for $h > 0$. 
\end{lem}

\begin{proof}
Due to the continuity of the Fourier transform $\cF: \Sigma_t^s (\rr d) \to \Sigma_s^t (\rr d)$ we have: 
For every $h_1 > 0$ there exists $h_2 > 0$ such that for $f \in \Sigma_t^s(\rr d)$
\begin{equation}\label{eq:GSFourier}
\nm {\wh f}  {\mathcal S_{s,h_1}^t}
\lesssim 
\nm f  {\mathcal S_{t,h_2}^s}. 
\end{equation}

Set for $r > 0$
\begin{equation*}
\| f \|_{t,r}' = \sup_{x \in \rr d} e^{r |x|^{\frac1t} } |f (x)|
\end{equation*}
and for $\fy \in \Sigma_t^s(\rr d) \setminus 0$
\begin{equation*}
\| f \|_r'' = \sup_{(x,\xi) \in \rr {2d}} e^{r \left( |x|^{\frac1t} + |\xi|^{\frac1s} \right) } |V_\fy f (x,\xi)|. 
\end{equation*}

We start by showing
\begin{equation}\label{eq:GSffhat2orig}
\forall r > 0 \ \exists h > 0: \ 
\| f \|_{t,r}' + \| \wh f \|_{s,r}' \lesssim \nm f  {\mathcal S_{t,h}^s}, \quad f \in \Sigma_t^s(\rr d). 
\end{equation}

Using 
\begin{equation*}
|x|^n \leqs d ^{\frac{n}{2}} \max_{|\alpha| = n} |x^\alpha|, \quad x \in \rr d, 
\end{equation*}
we obtain for $f \in \Sigma_t^s(\rr d)$, for any $r > 0$ and any $h > 0$,
\begin{align*}
e^{\frac{r}{t} |x|^{\frac{1}{t}}} \left| f(x) \right|^{\frac{1}{t}}
& = 
\sum_{n=0}^{\infty} 2^{-n} n!^{-1} \left( \frac{2 r}{t} |x|^{\frac{1}{t}} \right)^n 
\left| f(x) \right|^{\frac{1}{t}} \\
& \leqs 
2 \left( \sup_{n \geqs 0}  n!^{-t} \left( \left( \frac{2 r}{ t} \right)^t |x| \right)^n  \left| f (x) \right| \right)^{\frac{1}{t}} \\
& \lesssim 
\left( \sup_{n \geqs 0}  \left( \left( \frac{2 r}{t} \right)^t d^{\frac{1}{2}} \right)^n \max_{|\alpha| = n} 
\frac{\left| x^\alpha f(x) \right|}{n!^t} \right)^{\frac{1}{t}} \\
& \leqs
\nm f  {\mathcal S_{t,h}^s}^{\frac1t} 
\left( \sup_{n \geqs 0}  \left( \left( \frac{2 r}{t} \right)^t d^{\frac{1}{2}} h \right)^n \right)^{\frac{1}{t}} \\
& \lesssim
\nm f  {\mathcal S_{t,h}^s}^{\frac1t}, \quad x \in \rr d, 
\end{align*}
provided $h = h(r,t,d) > 0$ is sufficiently small. 
This shows 
\begin{equation*}
\forall r > 0 \ \exists h > 0: \ 
\| f \|_{t,r}' \lesssim \nm f  {\mathcal S_{t,h}^s}. 
\end{equation*}

As a byproduct, since $\wh f \in \Sigma_s^t(\rr d)$, this gives using \eqref{eq:GSFourier}
the following conclusion. 
If $f \in \Sigma_t^s(\rr d)$ and $r > 0$ then there exist $h_1, h_2 > 0$ such that 
\begin{align*}
\| f \|_{t,r}' + \| \wh f \|_{s,r}' \lesssim \nm f  {\mathcal S_{t,h_1}^s} + \nm {\wh f}  {\mathcal S_{s,h_1}^t}
\lesssim \nm f  {\mathcal S_{t,h_2}^s}. 
\end{align*}
We have proved \eqref{eq:GSffhat2orig}. 

Next we show the opposite estimate, that is
\begin{equation}\label{eq:GSorig2ffhat}
\forall h > 0 \ \exists r > 0: \ \nm f  {\mathcal S_{t,h}^s}
\lesssim \| f \|_{t,r}' + \| \wh f \|_{s,r}', \quad f \in \Sigma_t^s(\rr d). 
\end{equation}
The argument is quite long. 
It resembles the proof of \cite[Theorem~6.1.6]{Nicola1}. For completeness' sake we give the details. 

First we deduce two estimates that are needed. From \eqref{eq:GSffhat2orig} it follows that $\| f \|_{t,r}' < \infty$ and 
$\| \wh f \|_{s,r}' < \infty$ for any $r > 0$ when $f \in \Sigma_t^s(\rr d)$. 
Thus for any $r > 0$ we have 
\begin{align*}
\sum_{n=0}^\infty \frac{|x|^{\frac{n}{t}} |f(x)|^{\frac1t}}{n!} \left(\frac{r}{t} \right)^n
& = e^{ \frac{r}{t} |x|^{\frac1t} } |f(x)|^{\frac1t} \leqslant (\| f \|_{t,r}')^{\frac1t}, \quad x \in \rr d, 
\end{align*}
which gives the estimate
\begin{equation*}
|x|^{n} |f(x)| 
\leqs \| f \|_{t,r}' (n!)^t \left(\frac{t}{r} \right)^{t n}, \quad x \in \rr d, \quad n \in \no, 
\end{equation*}
and further
\begin{equation*}
|x^\alpha f(x)| 
\leqs \| f \|_{t,r}' (\alpha!)^t \left( \frac{d t}{r}\right)^{t|\alpha|}, \quad x \in \rr d, \quad \alpha \in \nn d. 
\end{equation*}

Finally we take the $L^2$ norm and estimate for an integer $k > d/4$ with $\ep = 4k - d > 0$:
\begin{equation}\label{eq:L2est1}
\begin{aligned}
\| x^\alpha f \|_{L^2} 
& \lesssim \sup_{x \in \rr d} \eabs{x}^{\frac{d+\ep}{2}} |x^\alpha f(x)|
\lesssim \sup_{x \in \rr d, \ |\gamma| \leqslant 2 k} |x^{\alpha+\gamma} f(x)| \\
& \leqslant \| f \|_{t,r}' ((\alpha+\gamma)!)^t  \left( \frac{d t}{r}\right)^{t |\alpha+\gamma|} \\
& \lesssim \| f \|_{t,r}' \alpha!^t \left( \frac{2 d t}{r}\right)^{t |\alpha|}, \quad \alpha \in \nn d, 
\end{aligned}
\end{equation}
using $(\alpha+\gamma)! \leqslant 2^{|\alpha+\gamma|} \alpha! \gamma!$ (cf. \cite{Nicola1}) and considering $k$ a fixed parameter. 

From \eqref{eq:L2est1}, $\| \wh f \|_{s,r}' < \infty$ for any $r > 0$, and Parseval's theorem we obtain
\begin{equation}\label{eq:L2est2}
\| D^\beta f \|_{L^2} 
= \| \xi^\beta \wh f \|_{L^2} 
\lesssim \| \wh f \|_{s,r}' \beta!^s \left( \frac{2 d s}{r}\right)^{s|\beta|}, \quad \beta \in \nn d. 
\end{equation}

We now start to prove \eqref{eq:GSorig2ffhat}.  
It suffices to assume $0 < h \leqs1 $. 
We have for $\alpha,\beta \in \nn d$ arbitrary and $f \in \Sigma_t^s(\rr d)$, 
using the Cauchy--Schwarz inequality, Parseval's theorem and the Leibniz rule
\begin{equation}\label{eq:intermediateestimate1}
\begin{aligned}
|x^\alpha D^\beta f(x)| 
& = (2\pi)^{-\frac{d}{2}} \left| \int_{\rr d} \widehat{x^\alpha D^\beta f} (\xi) e^{i \la x, \xi \ra} \dd \xi \right|
\lesssim \| \eabs{\cdot}^{\frac{d+\ep}{2}} \widehat{x^\alpha D^\beta f} \|_{L^2} \\
& \lesssim \max_{|\gamma| \leqslant 2k} \| D^\gamma (x^\alpha D^\beta f) \|_{L^2} \\
& \lesssim \max_{|\gamma| \leqslant 2k} \sum_{\mu \leqslant \min(\alpha,\gamma) } \binom{\gamma}{\mu} \binom{\alpha}{\mu} \mu! \| x^{\alpha-\mu} D^{\beta + \gamma-\mu} f \|_{L^2}, \quad x \in \rr d. 
\end{aligned}
\end{equation}

In the next intermediate step we rewrite the expression for the $L^2$ norm squared using integration by parts and estimate it as
\begin{align*}
& \| x^{\alpha-\mu} D^{\beta + \gamma-\mu} f \|_{L^2}^2 \\
& = |(D^{\beta + \gamma-\mu} f , x^{2(\alpha-\mu)} D^{\beta + \gamma-\mu} f )| \\
& = |(f , D^{\beta + \gamma-\mu}  (x^{2(\alpha-\mu)} D^{\beta + \gamma-\mu} f) )| \\
& \leqs \sum_{\kappa \leqs \min(\beta+\gamma-\mu,2(\alpha-\mu))} \binom{\beta+\gamma-\mu}{\kappa} \binom{2(\alpha-\mu)}{\kappa} \kappa! |(x^{2(\alpha-\mu)-\kappa} f, D^{2(\beta + \gamma-\mu)-\kappa} f )| \\ 
& \leqs \sum_{\kappa \leqs \min(\beta+\gamma-\mu,2(\alpha-\mu))} \binom{\beta+\gamma-\mu}{\kappa} \binom{2(\alpha-\mu)}{\kappa} \kappa! \| x^{2(\alpha-\mu)-\kappa} f \|_{L^2}  \| D^{2(\beta + \gamma-\mu)-\kappa} f \|_{L^2}. 
\end{align*}

Set $\sigma = \max(t,s)$, $\tau = \min(t,s)$
and note that for $\delta > 0$ we have by \eqref{eq:expestimate0} for any $r > 0$
\begin{equation*}
\kappa!^{-\delta} 
\leqs C_{d,\delta,r,t,s} \left( \frac{2 d \tau}{r}\right)^{2 \sigma |\kappa|}, \quad \kappa \in \nn d.
\end{equation*}

From \eqref{eq:L2est1}, \eqref{eq:L2est2} and $\kappa! = \kappa!^{t + s -\delta}$ where $\delta = t + s - 1> 0$, we get if $r \geqs 2 d \sigma$
\begin{align*}
& \| x^{\alpha-\mu} D^{\beta + \gamma-\mu} f \|_{L^2}^2 \\
& \lesssim 2^{2 |\alpha+\beta|} 
\left( \frac{2 d \sigma}{r}\right)^{2 \tau | \alpha + \beta + \gamma - 2 \mu|}
\|  f \|_{t,r}'  \| \wh f \|_{s,r}' 
\sum_{\kappa \leqs \min(\beta+\gamma-\mu,2(\alpha-\mu))} \kappa! 
\left( \frac{2 d \tau}{r}\right)^{-2 \sigma |\kappa|} \\
& \qquad \qquad \qquad \qquad \qquad \qquad \qquad \qquad \times
\left( ((2(\alpha-\mu)-\kappa)! \right)^t  \left((2(\beta + \gamma -\mu)-\kappa)! \right)^s \\
& \lesssim 2^{2 |\alpha+\beta|} 
\left( \frac{r}{2 d \sigma}\right)^{4 \tau |\mu|}
\left( \frac{2 d \sigma}{r}\right)^{2 \tau | \alpha + \beta|}
 \|  f \|_{t,r}'  \| \wh f \|_{s,r}' \sum_{\kappa \leqs \min(\beta+\gamma-\mu,2 (\alpha-\mu))}
 (\kappa!)^{-\delta} 
 \left( \frac{2 d \tau}{r}\right)^{-2 \sigma |\kappa|} \\
& \qquad \qquad \qquad \qquad \qquad \qquad \qquad \qquad \times
((2(\alpha-\mu))! )^t ((2(\beta + \gamma-\mu))!)^s \\
& \lesssim 2^{2 |\alpha+\beta|} 
\left( \frac{r}{2 d \sigma}\right)^{8 \sigma k}
\left( \frac{2 d \sigma}{r}\right)^{2 \tau | \alpha + \beta|}
 \|  f \|_{t,r}'  \| \wh f \|_{s,r}' \sum_{\kappa \leqs \min(\beta+\gamma-\mu,2 (\alpha-\mu))}
 ((2(\alpha-\mu))! )^t ((2(\beta + \gamma-\mu))!)^s \\
& \leqs C_{r,k} \, 2^{4 |\alpha+\beta|} 
\left( \frac{2 d \sigma}{r}\right)^{2 \tau | \alpha + \beta|}
 \|  f \|_{t,r}'  \| \wh f \|_{s,r}' 
 ((2(\alpha-\mu))! )^t ((2(\beta + \gamma-\mu))!)^s. 
\end{align*}

We insert this into \eqref{eq:intermediateestimate1} which gives, using $\mu! \leqslant \mu!^{t + s}$ and 
$$
(2(\alpha-\mu))! \leqslant 2^{2|\alpha|} ((\alpha-\mu)!)^2,
$$
\begin{align*}
& |x^\alpha D^\beta f(x)| \\
& \lesssim 2^{2 |\alpha+\beta|} 
\left( \frac{2 d \sigma}{r}\right)^{\tau | \alpha + \beta|}
\left( \|  f \|_{t,r}'  \| \wh f \|_{s,r}' \right)^{\frac12} \max_{|\gamma| \leqslant 2k} \sum_{\mu \leqs \min(\alpha,\gamma) } \binom{\gamma}{\mu} \binom{\alpha}{\mu} \mu!  ((2(\alpha-\mu))!)^{\frac{t}{2}} ( (2(\beta-\mu))!))^{\frac{s}{2}} \\
& \lesssim 2^{(2 + t + s) |\alpha+\beta|} 
\left( \frac{2 d \sigma}{r}\right)^{\tau | \alpha + \beta |}
\left( \|  f \|_{t,r}'  \| \wh f \|_{s,r}' \right)^{\frac12} \max_{|\gamma| \leqslant 2k} \sum_{\mu \leqs \min(\alpha,\gamma) } \binom{\alpha}{\mu} \alpha!^t \beta!^s \\
& \lesssim \left( 2^{3 + t + s} \left( \frac{2 d \sigma}{r} \right)^\tau \right)^{|\alpha+\beta|}  
\alpha!^t \beta!^s ( \|  f \|_{t,r}'  + \| \wh f \|_{s,r}' ), \quad x \in \rr d, \quad \alpha,\beta \in \nn d.
\end{align*}
Given $0 < h \leqs 1$ we may pick $r \geqs 2 d \sigma$ such that 
\begin{equation*}
h = 2^{3 + t + s} \left( \frac{2 d \sigma}{r} \right)^\tau. 
\end{equation*}
This finally proves \eqref{eq:GSorig2ffhat}. 

Next we use \eqref{eq:GSffhat2orig} in order to prove
\begin{equation}\label{eq:GSSTFT2orig}
\forall r > 0 \ \exists h > 0: \ 
\| f \|_r'' \lesssim \nm f  {\mathcal S_{t,h}^s}, \quad f \in \Sigma_t^s(\rr d).
\end{equation}

Let $r > 0$ and $\fy \in \Sigma_t^s(\rr d) \setminus 0$, and set $\kappa = \max (\kappa(t^{-1}), \kappa(s^{-1}) )$. 
Then \eqref{eq:STFTGFstfunc} and \eqref{eq:exppeetre} give for $f \in \Sigma_t^s(\rr d)$
\begin{align*}
|V_\fy f (x,\xi)| 
& = |\wh{f T_x \overline{\fy}} (\xi)|
\lesssim |\wh{f}| * |\wh{T_x \overline{\fy}}| (\xi)
= \int_{\rr d} |\wh{f}(\xi-\eta)| \, |\wh{\fy}(-\eta)| \, \dd \eta \\
& \lesssim \| \wh f \|_{s,2r \kappa}'  \| \wh \fy \|_{s,3r \kappa}' \int_{\rr d} e^{-2r \kappa |\xi-\eta|^{\frac1s} - 3 r \kappa |\eta|^{\frac1s} }  \, \dd \eta \\
& \lesssim \| \wh f \|_{s,2r \kappa}' \,  e^{-2r|\xi|^{\frac1s}} \int_{\rr d} e^{ (2 - 3 ) r \kappa |\eta|^{\frac1s}} \, \dd \eta \\
& \lesssim \| \wh f \|_{s,2r \kappa}'  \, e^{-2r|\xi|^{\frac1s}}, \quad x, \xi \in \rr d. 
\end{align*}
From this estimate and $|V_\fy f (x,\xi)| = |V_{\wh \fy} \wh f (\xi,-x)|$ we also obtain 
\begin{equation*}
|V_\fy f (x,\xi)|  
\lesssim \| f \|_{t,2r \kappa}'  e^{-2r |x|^{\frac1t}}, \quad x, \xi \in \rr d. 
\end{equation*}

We may conclude
\begin{align*}
e^{2 r \left( |x|^{\frac1t} +  |\xi|^{\frac1s} \right)} |V_\fy f (x,\xi)|^2
& = e^{2 r |x|^{\frac1t}} |V_\fy f (x,\xi)| \ e^{2 r |\xi|^{\frac1s}} |V_\fy f (x,\xi)| \\
& \lesssim \| f \|_{t,2r \kappa}'  \ \| \wh f \|_{s,2r \kappa}'
\end{align*}
which gives 
\begin{equation*}
\| f \|_r'' \lesssim \left( \| f \|_{t,2r \kappa}' \ \| \wh f \|_{s,2r \kappa}' \right)^{\frac12} \lesssim \| f \|_{t,2r \kappa}' + \| \wh f \|_{s,2r \kappa}'.
\end{equation*}
Combining with \eqref{eq:GSffhat2orig} we have proved \eqref{eq:GSSTFT2orig}. 

It remains to prove 
\begin{equation}\label{eq:GSorig2STFT}
\forall h > 0 \ \exists r > 0: \ 
\nm f  {\mathcal S_{t,h}^s}
\lesssim \| f \|_r'',  \quad f \in \Sigma_t^s(\rr d).
\end{equation}
which we do by means of \eqref{eq:GSorig2ffhat}. 

We use the strong version of the STFT inversion formula \eqref{eq:STFTinverse} and its Fourier transform, that is
\begin{align}
f(x) & = (2 \pi)^{-\frac{d}{2}} \int_{\rr {2d}} V_\fy f(y,\eta) M_\eta T_y \fy (x) \, \dd y \, \dd \eta, \label{eq:STFTinv1} \\
\wh f(\xi) & = (2 \pi)^{-\frac{d}{2}} \int_{\rr {2d}} V_\fy f(y,\eta) T_\eta M_{-y} \wh \fy (\xi) \, \dd y \, \dd \eta, \label{eq:STFTinv2}
\end{align}
where $f \in \Sigma_t^s(\rr d)$ and $\fy \in \Sigma_t^s(\rr d)$ satisfies $\| \fy \|_{L^2} = 1$. 

Set again $\kappa = \max (\kappa(t^{-1}), \kappa(s^{-1}) )$. 
From \eqref{eq:GSffhat2orig} and \eqref{eq:STFTinv1} we obtain for any $r > 0$ 
\begin{align*}
e^{r |x|^{\frac1t}} |f(x)| 
& \lesssim \int_{\rr {2d}} |V_\fy f(y,\eta)| \, e^{r |x|^{\frac1t}}  |\fy (x-y)| \, \dd y \, \dd \eta \\
& \lesssim \| f \|_{2r \kappa}''  \int_{\rr {2d}} e^{-2r \kappa \left( |y|^{\frac1t} + |\eta|^{\frac1s} \right)} \, e^{r |x|^{\frac1t} - r \kappa |x-y|^{\frac1t}} \dd y \, \dd \eta, \\
& \lesssim \| f \|_{2 r \kappa}''  \int_{\rr {2d}} e^{- 2r \kappa \left( |y|^{\frac1t} + |\eta|^{\frac1s} \right)} \, e^{r \kappa |y|^{\frac1t}} \dd y \, \dd \eta, \\
& \lesssim \| f \|_{2 r \kappa}'', \quad x \in \rr d, 
\end{align*}
which gives $\| f \|_{t,r}' \lesssim \| f \|_{2 r \kappa}''$. 

From \eqref{eq:STFTinv2} we obtain for any $r > 0$
\begin{align*}
e^{r |\xi|^{\frac1s}} |\wh f(\xi)| 
& \lesssim \int_{\rr {2d}} |V_\fy f(y,\eta)| \, e^{r |\xi|^{\frac1s}} |\wh \fy (\xi-\eta)| \, \dd y \, \dd \eta, \\
& \lesssim \| f \|_{2r \kappa}''  \int_{\rr {2d}} e^{-2r \kappa \left( |y|^{\frac1t} + |\eta|^{\frac1s} \right)} \, e^{r |\xi|^{\frac1s} - r \kappa |\xi-\eta|^{\frac1s}}  \, \dd y \, \dd \eta, \\
& \lesssim \| f \|_{2r \kappa}'', \quad \xi \in \rr d, 
\end{align*}
which gives $\| \wh f \|_{s,r}' \lesssim \| f \|_{2 r \kappa}''$. Thus $\| f \|_{t,r}' + \| \wh f \|_{s,r}' \lesssim \| f \|_{2 r \kappa}''$ so combining with \eqref{eq:GSorig2ffhat} we have proved \eqref{eq:GSorig2STFT}. 

Finally we note that the seminorms $\{\| f \|_{r}'', \ r >0 \}$ are equivalent to the same family of seminorms when the window function $\fy \in \Sigma_t^s(\rr d) \setminus 0$ is replaced by another function $\psi \in \Sigma_t^s(\rr d) \setminus 0$.
Indeed this is an immediate consequence of \eqref{eq:GSSTFT2orig} and \eqref{eq:GSorig2STFT}.
\end{proof}

\section{Anisotropic Gelfand--Shilov wave front sets}\label{sec:anisotropicGSWF}

\subsection{$s$-conic subsets}

For $s > 0$ we use subsets of $T^* \rr d \setminus 0$ that are $s$-conic, 
that is subsets closed under the operation $T^* \rr d \setminus 0 \ni (x,\xi) \mapsto ( \lambda x, \lambda^s \xi)$
for all $\lambda > 0$.
Thus $1$-conic is the same as the usual definition of conic. 

Let $t, s > 0$ be fixed. 
We need the following simplified version of a tool taken from \cite{Parenti1} and its references. 
Given $(x,\xi) \in \rr {2d} \setminus 0$ there is a unique $\lambda = \lambda(x,\xi) = \lambda_{t,s} (x,\xi) > 0$ such that 
\begin{equation*}
\lambda (x,\xi)^{-2t} | x |^2 + \lambda (x,\xi)^{-2s} | \xi |^2 = 1. 
\end{equation*}
Then $(x,\xi) \in \sr {2d-1}$ if and only if $\lambda (x,\xi) = 1$. 
By the implicit function theorem the function $\lambda: \rr {2d} \setminus 0 \to \ro_+$ is smooth \cite{Krantz1}. 

If $\mu > 0$ and $(x,\xi) \in \sr {2d-1}$ then $\lambda ( \mu^t x, \mu^s \xi) = \mu = \mu \lambda (x,\xi)$. 
In fact 
\begin{equation}\label{eq:quasihomogen1}
\lambda ( \mu^t x, \mu^s \xi) = \mu \lambda (x,\xi)
\end{equation}
holds for any $(x,\xi) \in \rr {2d} \setminus 0$ and $\mu > 0$ by the following argument. 
Given $(x,\xi) \in \rr {2d} \setminus 0$ set $\mu_1 = \lambda (x,\xi)$ so that 
$(x/ \mu_1^t, \xi/\mu_1^s) \in \sr {2d-1}$. 
Then for $\mu > 0$
\begin{equation*}
\lambda ( \mu^t x, \mu^s \xi) = \lambda ( (\mu \mu_1)^t x/\mu_1^t, (\mu \mu_1)^s \xi/\mu_1^s ) 
= \mu \mu_1 
= \mu \lambda (x,\xi). 
\end{equation*}

The projection $p(x,\xi) = p_{t,s}(x,\xi)$ of $(x,\xi) \in \rr {2d} \setminus 0$ along the curve $\ro_+ \ni \mu \mapsto (\mu^t x, \mu^s \xi)$ onto $\sr {2d-1}$ is defined as
\begin{equation}\label{eq:projection}
p(x,\xi) = \left( \lambda(x,\xi)^{-t} x, \lambda(x,\xi)^{-s} \xi \right), \quad (x,\xi) \in \rr {2d} \setminus 0. 
\end{equation}
Then $p(\mu^t x, \mu^s \xi) = p(x, \xi)$ does not depend on $\mu > 0$. 
The function $p: \rr {2d} \setminus 0 \to \sr {2d-1}$ is smooth since $\lambda \in C^\infty(\rr {2d} \setminus 0)$
and $\lambda(x,\xi) > 0$ for all $(x,\xi) \in \rr {2d} \setminus 0$. 

Note that $\lambda_{t,s} (x,\xi) = \lambda_{1,\frac{s}{t}} (x,\xi)^{\frac1t}$, and thus $p_{t,s}(x,\xi) = p_{1,\frac{s}{t}}(x,\xi)$
depends only on $\frac{s}{t}$. 

From \cite{Parenti1}, or by straightforward arguments, we have the bounds
\begin{equation*}
|x|^{\frac1t} + |\xi|^{\frac1s}
\lesssim \lambda(x,\xi) \lesssim |x|^{\frac1t} + |\xi|^{\frac1s}, \quad (x,\xi) \in \rr {2d} \setminus 0, 
\end{equation*}
and
\begin{equation*}
\eabs{ (x,\xi) }^{\min \left( 1, \frac1t \right) \min \left( 1, \frac1s \right)}
\lesssim 1 + \lambda(x,\xi) 
\lesssim \eabs{(x,\xi)}^{\max \left( 1, \frac1t \right) \max \left( 1, \frac1s \right)}, \quad (x,\xi) \in \rr {2d} \setminus 0. 
\end{equation*}

We will use two types of $s$-conic neighborhoods. 
The first type is defined as follows. 

\begin{defn}\label{def:scone1}
Suppose $s, \ep > 0$ and $z_0 \in \sr {2d-1}$. 
Then
\begin{equation*}
\Gamma_{s, z_0, \ep}
= \{ (x,\xi) \in \rr {2d} \setminus 0, \ | z_0 - p_{1,s} (x,\xi) | < \ep \}
\subseteq  T^* \rr d \setminus 0. 
\end{equation*}
\end{defn}

We write $\Gamma_{z_0, \ep} = \Gamma_{s, z_0, \ep}$ when $s$ is fixed and understood from the context. 
If $\ep > 2$ then $\Gamma_{z_0, \ep} = T^* \rr d \setminus 0$ so we usually restrict to $\ep \leqs 2$.

The second type of $s$-conic neighborhood is defined as follows. 

\begin{defn}\label{def:scone2}
Suppose $s, \ep > 0$ and $(x_0, \xi_0) \in \sr {2d-1}$. 
Then
\begin{align*}
\wt \Gamma_{(x_0,\xi_0),\ep}
= \wt \Gamma_{s,(x_0,\xi_0),\ep} 
& = \{ (y,\eta) \in \rr {2d} \setminus 0: \ (y,\eta) = (\lambda (x_0 + x), \lambda^s (\xi_0 + \xi), \ \lambda > 0, \ (x,\xi) \in \rB_\ep \} \\
& = \{ (y,\eta) \in \rr {2d} \setminus 0: \ \exists \lambda > 0: \ (\lambda y, \lambda^s \eta) \in (x_0,\xi_0) + \rB_\ep \}.  
\end{align*}
\end{defn}

By \cite[Lemma~3.7]{Rodino4} the two types of $s$-conic neighborhoods are equivalent. 
This means that if $z_0 \in \sr {2d-1}$ then
for each $\ep > 0$ there exists $\delta > 0$ such that 
$\Gamma_{z_0, \delta} \subseteq \widetilde \Gamma_{z_0,\ep}$
and
$\wt \Gamma_{z_0, \delta} \subseteq \Gamma_{z_0,\ep}$.

\subsection{Anisotropic Gelfand--Shilov wave front sets}

For $u \in (\Sigma_t^s)' (\rr d)$ the $t,s$-Gelfand--Shilov wave front set $\WF^{t,s} (u)$ was defined in \cite{Rodino3} as a closed subset of the phase space $T^* \rr d \setminus 0$ as follows. 

\begin{defn}\label{def:wavefrontGFst}
Let $s,t > 0$ satisfy $s + t > 1$, and suppose $\psi \in \Sigma_t^s(\rr d) \setminus 0$ and $u \in (\Sigma_t^s)'(\rr d)$. 
Then $(x_0,\xi_0) \in T^* \rr d \setminus 0$ satisfies $(x_0,\xi_0) \notin \WF^{t,s} (u)$ if there exists an open set $U \subseteq T^*\rr d \setminus 0$ containing $(x_0,\xi_0)$ such that 
\begin{equation}\label{eq:notinWFGFst1}
\sup_{\lambda > 0, \ (x,\xi) \in U} e^{r \lambda} |V_\psi u(\lambda^t x, \lambda^s \xi)| < \infty, \quad \forall r > 0. 
\end{equation}
\end{defn}

Due to \eqref{eq:STFTGFstdistr} it is clear that it suffices to check \eqref{eq:notinWFGFst1} for $\lambda \geqs L$
where $L > 0$ can be arbitrarily large, for each $r > 0$. 

A consequence of Definition \ref{def:wavefrontGFst} is that $\WF^{t,s} (u)$ is an $\frac{s}{t}$-conic closed subset of $T^* \rr d \setminus 0$. 
If $t = s > \frac12$ and $u \in \Sigma_s' (\rr d)$ then $\WF^{s,s} (u) = \WF^s(u)$, so we recapture 
the $s$-Gelfand--Shilov wave front set $\WF^s (u)$ (which is a slightly modified version of Cappiello's and Schulz's \cite[Definition~2.1]{Cappiello1}),
as defined originally in \cite[Definition~4.1]{Carypis1}: 

\begin{defn}\label{def:wavefrontGFs}
Let $s > 1/2$, $\psi \in \Sigma_s(\rr d) \setminus 0$ and $u \in \Sigma_s'(\rr d)$. 
Then $z_0 \in T^*\rr d \setminus 0$ satisfies $z_0 \notin \WF^s (u)$ if there exists an open conic set $\Gamma_{z_0} \subseteq T^*\rr d \setminus 0$ containing $z_0$ such that 
\begin{equation*}
\sup_{z \in \Gamma_{z_0}} e^{r | z |^{\frac1s}} |V_\psi u(z)| < \infty, \quad \forall r > 0. 
\end{equation*}
\end{defn}

In Definition \ref{def:wavefrontGFst} we ask for exponential decay with arbitrary parameter $r > 0$ (super-exponential) of $V_\psi u$ along the curve $C_{x,\xi} \in T^* \rr d$
defined by $\ro_+ \ni  \lambda \to (\lambda^t x, \lambda^s \xi)$ which passes through $(x,\xi) \in U \subseteq T^* \rr d \setminus 0$. 
This power type curve reduces to a straight line if $t=s$.
By \eqref{eq:STFTGFstdistr} a generic point $(x,\xi) \in T^* \rr d \setminus 0$ has an exponential growth upper bound along the curve $C_{x,\xi}$. 
Due to \eqref{eq:STFTGFstfunc} we have $\WF^{t,s} (u) = \emptyset$ if and only if $u \in \Sigma_t^s(\rr d)$. 
Thus $\WF^{t,s} (u) \subseteq T^* \rr d \setminus 0$ can be seen as a measure of singularities of $u \in (\Sigma_t^s)'(\rr d)$: It records the phase space points $(x,\xi) \in T^* \rr d \setminus 0$ such that $V_\psi u$ does not decay super-exponentially along the curve $C_{x,\xi}$, 
that is, does not behave like an element in $\Sigma_t^s(\rr d)$ there. 

The $t,s$-Gelfand--Shilov wave front set is related to the anisotropic Gabor wave front set 
\cite{Rodino4,Zhu1} where the functional framework is the Schwartz space and tempered distributions, 
and the decay and growth in phase space are polynomial rather than exponential. 

By \cite[Proposition~3.5]{Rodino3}, Definition \ref{def:wavefrontGFst} does not depend on the window function $\psi \in \Sigma_t^s(\rr d) \setminus 0$. 
If $\check u(x) = u(-x)$ then 
\begin{equation}\label{eq:evensteven0}
V_{\check \psi} \check u(x,\xi)
= V_\psi u(-x,-\xi). 
\end{equation}
If $u$ is even or odd we thus have the following symmetry:
\begin{equation}\label{eq:evensteven1}
\check u = \pm u \quad \Longrightarrow \quad \WF^{t,s} (u) = - \WF^{t,s} (u). 
\end{equation}

We also have 
\begin{equation}\label{eq:evensteven2}
V_{\psi} \overline{u}(x,\xi)
= \overline{V_{\overline \psi} u(x,-\xi)}. 
\end{equation}

By \cite[Remark~3.4]{Rodino3} we have $\WF^{t p ,s p} (u) \subseteq \WF^{t,s} (u)$
if $p \geqs 1$, $t + s > 1$ and $u \in (\Sigma_{t p}^{s p})'(\rr d)$. 

If $(y,\eta) \in \wt \Gamma_{\frac{s}{t},(x_0,\xi_0), \ep}$ for $0 < \ep < 1$ then for some $\lambda > 0$
and $(x,\xi) \in \rB_\ep$ we have
$(y,\eta) = (\lambda^t (x_0+x), \lambda^s (\xi_0+\xi))$. 
Thus $|y|^{\frac1t} + |\eta|^{\frac1s} \asymp \lambda$,
which gives 
the following equivalent criterion to the condition \eqref{eq:notinWFGFst1} in Definition \ref{def:wavefrontGFst}. 
The point $(x_0,\xi_0) \in \sr {2d-1}$ satisfies $(x_0,\xi_0) \notin \WF^{t,s} (u)$ if and only if for some $\ep > 0$
we have
\begin{equation}\label{eq:WFgs2}
\sup_{(x,\xi) \in \Gamma_{\frac{s}{t}, (x_0,\xi_0), \ep}} e^{r \left( |x|^{\frac1t} + |\xi|^{\frac1s} \right) } |V_\fy u (x,\xi)| < + \infty \quad \forall r > 0. 
\end{equation}

We will use the following result on the anisotropic Gelfand--Shilov wave front set of a tensor product. 
Here we use the notation $x=(x',x'') \in \rr {m+n}$, $x' \in \rr m$, $x'' \in \rr n$. 

\begin{prop}\label{prop:tensorWFs}
If $t,s > 0$, $t + s > 1$, $u \in \left(\Sigma_t^s \right)' (\rr m)$, and $v \in \left(\Sigma_t^s \right)' (\rr n)$ then 
\begin{align*}
& \WF^{t,s} (u \otimes v) \subseteq \left( ( \WF^{t,s} (u) \cup \{0\} ) \times ( \WF^{t,s}(v) \cup \{0\} ) \right)\setminus 0 \\
& = \{ (x,\xi) \in T^* \rr {m+n} \setminus 0: \ (x',\xi') \in \WF^{t,s}(u) \cup \{ 0 \}, \ (x'',\xi'') \in \WF^{t,s}(v) \cup \{ 0 \} \} \setminus 0. 
\end{align*}
\end{prop}

\begin{proof}
Let $\fy \in \Sigma_t^s(\rr m) \setminus 0$ and $\psi \in \Sigma_t^s(\rr n) \setminus 0$. 
Suppose $(x_0,\xi_0) \in T^* \rr {m+n} \setminus 0$ does not belong to the set on the right hand side. 
Then either $(x_0',\xi_0') \notin \WF^{t,s}(u) \cup \{ 0 \}$ or $(x_0'',\xi_0'') \notin \WF^{t,s}(v) \cup \{ 0 \}$. 
For reasons of symmetry we may assume $(x_0',\xi_0') \notin \WF^{t,s}(u) \cup \{ 0 \}$. 

Thus there exists $\ep > 0$ such that 
\begin{equation*}
\sup_{(x',\xi') \in (x_0',\xi_0') + \rB_\ep, \ \lambda > 0} e^{ r \lambda } |V_\fy u( \lambda^t x', \lambda^s \xi')| < \infty \quad \forall r > 0. 
\end{equation*}

Let $(x',\xi') \in (x_0',\xi_0') + \rB_\ep$, $(x'',\xi'') \in (x_0'',\xi_0'') + \rB_\ep$, 
let $r > 0$ be arbitrary, and let $\lambda \geqs 1$. 
We obtain using \eqref{eq:STFTGFstdistr}, for some $r_1 > 0$
\begin{align*}
e^{ r \lambda } |V_{\fy \otimes \psi} u \otimes v (\lambda^t x, \lambda^s \xi)|
& = e^{ r \lambda } |V_\fy u ( \lambda^t x', \lambda^s \xi')| \, |V_\psi v (\lambda^t x'', \lambda^s \xi'')| \\
& \lesssim e^{ (r + r_1) \lambda}  |V_\fy u ( \lambda^t x', \lambda^s \xi')| < \infty. 
\end{align*}
It follows that $(x_0,\xi_0) \notin \WF^{t,s} (u \otimes v)$. 
\end{proof}

\section{Propagation of anisotropic Gelfand--Shilov wave front sets}\label{sec:propagation}

Let $t,s > 0$ and $t + s > 1$.
Define for $K \in \left( \Sigma_t^s \right)'(\rr {2d})$
\begin{align*}
\WF_{1}^{t,s}(K) & = \{ (x,\xi) \in T^* \rr d: \ (x, 0, \xi, 0) \in \WF^{t,s} (K) \} \subseteq T^* \rr d \setminus 0, \\
\WF_{2}^{t,s}(K) & = \{ (y,\eta) \in T^* \rr d: \ (0, y, 0, -\eta) \in \WF^{t,s} (K) \} \subseteq T^* \rr d \setminus 0. 
\end{align*}

We will use the assumption
\begin{equation}\label{eq:WFKjempty}
\WF_{1}^{t,s}(K) = \WF_{2}^{t,s}(K) = \emptyset. 
\end{equation}  

It is clear that if \eqref{eq:WFKjempty} holds for $K \in \left( \Sigma_t^s \right)'(\rr {2d})$ then 
$\WF_{1}^{t p,s p}(K) = \WF_{2}^{t p,s p}(K) = \emptyset$
if $K \in \left( \Sigma_{t p}^{s p} \right)'(\rr {2d})$, for any $p \geqs 1$. 

The following lemma is an $\frac{s}{t}$-conic version of \cite[Lemma~6.1]{Carypis1} which treats the isotropic Gelfand--Shilov wave front set. 

\begin{lem}\label{lem:WFkernelaxes}
If $t,s > 0$, $K \in \left( \Sigma_t^s \right)' (\rr {2d})$ and \eqref{eq:WFKjempty} holds, then there exists $c > 1$ such that 
\begin{equation}\label{eq:Gamma1}
\WF^{t,s} (K) \subseteq 
\Gamma_1 := \left\{ (x,y,\xi,\eta) \in T^* \rr {2d}: \ c^{-1} \left( |x|^{\frac1t} + |\xi|^{\frac1s} \right) <  |y|^{\frac1t}  + |\eta|^{\frac1s}  < c \left( |x|^{\frac1t} + |\xi|^{\frac1s} \right)  \right\}. 
\end{equation}
\end{lem}

\begin{proof}
Suppose 
\begin{equation*}
\WF^{t,s} (K) \subseteq \left\{ (x,y,\xi,\eta) \in T^* \rr {2d}: \ |y|^{\frac1t}  + |\eta|^{\frac1s}  < c \left( |x|^{\frac1t}  + |\xi|^{\frac1s} \right) \right\}
\end{equation*}
does not hold for any $c > 0$. Then for each $n \in \no$ there exists $(x_n,y_n,\xi_n,\eta_n) \in \WF^{t,s} (K)$ such that 
\begin{equation}\label{eq:phasespaceineq1}
|y_n|^{\frac1t}  + |\eta_n|^{\frac1s} \geqs n \left( |x_n| ^{\frac1t} + |\xi_n|^{\frac1s} \right).
\end{equation}
By rescaling $(x_n,y_n,\xi_n,\eta_n)$ as
\begin{equation*}
(x_n,y_n,\xi_n,\eta_n) \mapsto ( \lambda^t x_n, \lambda^t y_n, \lambda^{s} \xi_n, \lambda^{s} \eta_n) 
\end{equation*}
we obtain for a unique $\lambda = \lambda (x_n,y_n,\xi_n,\eta_n) > 0$ a vector in $\sr {4d-1}$ \cite{Rodino4}. 
This $\frac{s}{t}$-conic rescaling leaves \eqref{eq:phasespaceineq1} invariant. 
Abusing notation we still denote the rescaled vector $(x_n,y_n,\xi_n,\eta_n) \in \WF^{t,s} (K) \cap \sr {4d-1}$.

From \eqref{eq:phasespaceineq1}  it follows that $(x_n,\xi_n) \rightarrow 0$ as $n \rightarrow \infty$. 
Passing to a subsequence (without change of notation) and using the closedness of $\WF^{t,s} (K)$ gives
\begin{equation*}
(x_n,y_n,\xi_n,\eta_n) \rightarrow (0,y,0,\eta) \in \WF^{t,s} (K), \quad n \rightarrow \infty, 
\end{equation*}
for some $(y,\eta) \in \sr {2d-1}$. This implies $(y,-\eta) \in \WF_{2}^{t,s} (K)$ which contradicts the assumption \eqref{eq:WFKjempty}. 

Similarly one shows 
\begin{equation*}
\WF^{t,s} (K) \subseteq \left\{ (x,y,\xi,\eta) \in T^* \rr {2d}: \  |x|^{\frac1t} + |\xi|^{\frac1s}  <  c \left( |y|^{\frac1t} + |\eta|^{\frac1s} \right) \right\}
\end{equation*}
for some $c > 0$ using $\WF_{1}^{t,s}(K) = \emptyset$. 
\end{proof}

The set $\Gamma_1 \subseteq \rr {4d} \setminus 0$ in \eqref{eq:Gamma1} is open, and $\frac{s}{t}$-conic in the sense that it is closed with respect
to $(x,y,\xi,\eta) \mapsto ( \lambda^t (x,y), \lambda^s (\xi, \eta))$ for any $\lambda > 0$. 
Hence $(\rr {4d} \setminus \Gamma_1)$ is $\frac{s}{t}$-conic and $(\rr {4d} \setminus \Gamma_1) \cap \sr{4d-1}$ is compact. 
From \eqref{eq:WFgs2} we then obtain if $\Phi \in \Sigma_t^s(\rr {2d}) \setminus 0$
\begin{equation}\label{eq:WFKcomplement}
| V_\Phi K( x, y, \xi, - \eta) |  
\lesssim e^{- r \left( |(x,y)|^{\frac1t} + |( \xi,\eta )|^{\frac1s} \right) } , \quad r > 0, \quad ( x, y, \xi, - \eta) \in \rr {4d} \setminus \Gamma_1. 
\end{equation}

A kernel $K \in \left( \Sigma_t^s \right)'(\rr {2d})$ defines a continuous linear map $\cK: \Sigma_t^s (\rr d) \to \left( \Sigma_t^s \right)'(\rr d)$ by
\begin{equation}\label{eq:kernelop}
(\cK f, g) = (K, g \otimes \overline f), \quad f,g \in \Sigma_t^s(\rr d). 
\end{equation}

The following result says that \eqref{eq:WFKjempty} implies continuity of $\cK$ on $\Sigma_t^s(\rr d)$
and admits a unique extension to a continuous operator on $(\Sigma_t^s)'(\rr d)$. 
This is the basis for the forthcoming result on propagation of the $t,s$-Gelfand--Shilov wave front sets Theorem \ref{thm:WFGSpropagation}. 
In the proof we use the conventional notation (cf. \cite{Hormander1,Hormander2}) for the reflection operator in the fourth $\rr d$ coordinate in $\rr {4d}$
\begin{equation}\label{eq:reflection}
(x,y,\xi,\eta)' = (x,y,\xi,-\eta), \quad x,y,\xi,\eta \in \rr d. 
\end{equation}

\begin{prop}\label{prop:opSTFTformula}
Let $t,s > 0$ satisfy $t + s > 1$, and let $\cK: \Sigma_t^s(\rr d) \to \left( \Sigma_t^s \right)'(\rr d)$ be the continuous linear operator \eqref{eq:kernelop}
defined by the Schwartz kernel $K \in \left( \Sigma_t^s \right)' (\rr {2d})$. 
If \eqref{eq:WFKjempty} holds then

\begin{enumerate}[\rm (i)]

\item $\cK: \Sigma_t^s(\rr d) \to \Sigma_t^s(\rr d)$ is continuous; 

\item $\cK$ extends uniquely to a continuous linear operator $\cK: \left( \Sigma_t^s \right)' (\rr d) \to \left( \Sigma_t^s \right)' (\rr d)$; 

\item if $\fy \in \Sigma_t^s(\rr d)$, $\| \fy \|_{L^2} = 1$, $\Phi = \fy \otimes \fy \in \Sigma_t^s(\rr {2d})$, $u \in \left( \Sigma_t^s \right)' (\rr d)$ and $\psi \in \Sigma_t^s(\rr d)$, then
\begin{equation}\label{eq:opSTFT1}
(\cK u, \psi) 
= \int_{\rr {4d}} V_\Phi K(x,y,\xi,-\eta) \, \overline{V_\fy \psi (x,\xi)} \, V_{\overline \fy} u(y,\eta) \, \dd x \, \dd y \, \dd \xi \, \dd \eta. 
\end{equation}

\end{enumerate}
\end{prop}

\begin{proof}
By \cite[Lemma~5.1]{Wahlberg1} the formula \eqref{eq:opSTFT1} holds for $u,\psi \in \Sigma_t^s(\rr d)$. 

Let $\fy \in \Sigma_t^s(\rr d)$ satisfy $\| \fy \|_{L^2} = 1$ and set $\Phi = \fy \otimes \fy \in \Sigma_t^s(\rr {2d})$. 
Since
\begin{equation*}
\overline{V_\fy \Pi(x,\xi) \fy (y,\eta)} = e^{i \la y, \eta - \xi \ra} V_\fy \fy ( x-y, \xi - \eta)
\end{equation*}
we get from \eqref{eq:opSTFT1} for $u \in \Sigma_t^s (\rr d)$ and $(x,\xi) \in T^* \rr d$
\begin{equation}\label{eq:STFTKu}
\begin{aligned}
& V_\fy(\cK u) (x, \xi) 
= (2 \pi)^{-\frac{d}{2}} (\cK u, \Pi(x,\xi) \fy) \\
& = (2 \pi)^{-\frac{d}{2}} \int_{\rr {4d}} e^{i \la y,\eta -\xi \ra} V_\Phi K (y,z,\eta,-\theta) V_\fy \fy (x-y,\xi-\eta) \, V_{\overline \fy} u(z,\theta) \, \dd y \, \dd z \, \dd \eta \, \dd \theta. 
\end{aligned}
\end{equation}

This gives 
\begin{equation}\label{eq:STFTKuabs}
|V_\fy(\cK u) (x, \xi)|
\lesssim \int_{\rr {4d}} | V_\Phi K (y,z,\eta,-\theta) | \, | V_\fy \fy (x-y,\xi-\eta)| \, | V_{\overline \fy} u(z,\theta)| \, \dd y \, \dd z \, \dd \eta \, \dd \theta. 
\end{equation}

We use the seminorms \eqref{eq:seminormsSigmats2}, denoted $\| \cdot \|_{r}''$ for $r > 0$ as in the proof of Lemma \ref{lem:Sigmatsseminorm}. 
Let $r > 0$, set $\kappa = \max (\kappa(t^{-1}), \kappa(s^{-1}) )$, and
consider first the right hand side integral of \eqref{eq:STFTKuabs} over $(y,z,\eta,-\theta) \in \rr {4d} \setminus \Gamma_1$
where $\Gamma_1$ is defined by \eqref{eq:Gamma1} with $c > 1$ chosen 
so that $\WF^{t,s} (K) \subseteq \Gamma_1$.
By Lemma \ref{lem:WFkernelaxes} we may use the estimates \eqref{eq:WFKcomplement}. 
Using \eqref{eq:STFTGFstfunc} and \eqref{eq:exppeetre} we obtain for any $r_1 > 0$
\begin{equation}\label{eq:STFTKuabs1}
\begin{aligned}
& \int_{\rr {4d} \setminus \Gamma_1'} 
|V_\Phi K(y,z,\eta,-\theta)| \,  | V_\fy \fy (x-y,\xi-\eta)|  \, |V_{\overline \fy} u (z,\theta)| \, \dd y \, \dd z \, \dd \eta \, \dd \theta \\
& \lesssim 
\int_{\rr {4d} \setminus \Gamma_1'} 
e^{-r_1 \left( |(y,z)|^{\frac1t} + |(\eta,\theta)|^{\frac1s} \right) } \,  
e^{- r \kappa \left( |x-y|^{\frac1t} + |\xi -\eta|^{\frac1s} \right) } \, 
|V_{\overline \fy} u (z,\theta)| \, \dd y \, \dd z \, \dd \eta \, \dd \theta \\
& \lesssim 
\| u \|_0'' \  e^{-r \left( |x|^{\frac1t} + |\xi|^{\frac1s} \right) }
\int_{\rr {4d}}  e^{ (r \kappa -r_1) \left( |(y,z)|^{\frac1t} + |(\eta,\theta)|^{\frac1s} \right) } \, \dd y \, \dd z \, \dd \xi \, \dd \eta \\
& \lesssim \| u \|_0'' \ e^{-r \left( |x|^{\frac1t} + |\xi|^{\frac1s} \right) }
\end{aligned}
\end{equation}
provided $r_1 > r \kappa$. 

Next we consider the right hand side integral \eqref{eq:STFTKuabs} over $(y,z,\eta,-\theta) \in \Gamma_1$. 
Then we may by Lemma \ref{lem:WFkernelaxes} 
use \eqref{eq:Gamma1}. 
Using \eqref{eq:STFTGFstdistr} and \eqref{eq:STFTGFstfunc} we obtain for some $r_1 > 0$
and any $r_2 > 0$
\begin{equation}\label{eq:STFTKuabs2}
\begin{aligned}
& \int_{\Gamma_1'} 
|V_\Phi K(y,z,\eta,-\theta)| \,  | V_\fy \fy (x-y,\xi-\eta)|  \, |V_{\overline \fy} u (z,\theta)| \, \dd y \, \dd z \, \dd \eta \, \dd \theta \\
& \lesssim 
\| u \|_{r_2}'' \  e^{-r \left( |x|^{\frac1t} + |\xi|^{\frac1s} \right) }
\int_{\Gamma_1'} 
e^{ r_1 \left( |(y,z)|^{\frac1t} + |(\eta,\theta)|^{\frac1s} \right) } \,  
e^{r \kappa \left( |y|^{\frac1t} + |\eta|^{\frac1s} \right) } \,
e^{-r_2 \left( |z|^{\frac1t} + |\theta|^{\frac1s} \right) } \, 
\dd x \, \dd y \, \dd \xi \, \dd \eta \\
& \leqs 
\| u \|_{r_2}'' \  e^{-r \left( |x|^{\frac1t} + |\xi|^{\frac1s} \right) }
\int_{\Gamma_1'} 
e^{ - \left( |(y,z)|^{\frac1t} + |(\eta,\theta)|^{\frac1s} \right) } \,  
e^{ (r_1 + 1) \kappa \left( |y|^{\frac1t} + |z|^{\frac1t} + |\eta|^{\frac1s} + |\theta|^{\frac1s} \right)} \\
& \qquad \qquad \qquad \qquad \qquad \qquad \qquad \qquad \qquad \qquad \qquad \times
e^{ \left(r \kappa c - r_2 \right) \left( |z|^{\frac1t} + |\theta|^{\frac1s} \right) } \, 
\dd x \, \dd y \, \dd \xi \, \dd \eta \\
& \leqs 
\| u \|_{r_2}'' \  e^{-r \left( |x|^{\frac1t} + |\xi|^{\frac1s} \right) }
\int_{\Gamma_1'} 
e^{ - \left( |(y,z)|^{\frac1t} + |(\eta,\theta)|^{\frac1s} \right) } \,  
e^{ \left((r_1 + 1) \kappa (1+c) + r \kappa c - r_2 \right) \left( |z|^{\frac1t} + |\theta|^{\frac1s} \right) } \, 
\dd x \, \dd y \, \dd \xi \, \dd \eta \\
& \lesssim 
\| u \|_{r_2}'' \  e^{-r \left( |x|^{\frac1t} + |\xi|^{\frac1s} \right) }
\end{aligned}
\end{equation}
provided $r_2 > 0$ is sufficiently large. 

Combining \eqref{eq:STFTKuabs1} and \eqref{eq:STFTKuabs2} we obtain from \eqref{eq:STFTKuabs}
$\| \cK u \|_r'' \lesssim \| u \|_{r_2}''$, which proves claim (i). 

To show claims (ii) and (iii) let $u \in \left( \Sigma_t^s \right)'(\rr d)$ and set for $n \in \no$
\begin{equation*}
u_n = (2 \pi)^{-\frac{d}{2}} \int_{|(y,\eta)| \leqs n} V_\fy u(y,\eta) \Pi(y,\eta) \fy \, \dd y \, \dd \eta. 
\end{equation*}
Let $r > 0$. 
From \eqref{eq:STFTGFstdistr} and \eqref{eq:STFTGFstfunc}
we obtain for some $r_1 > 0$ 
\begin{align*}
e^{r \left( |x|^{\frac1t} + |\xi|^{\frac1s} \right) }
|V_\fy u_n (x,\xi)| 
& \lesssim \int_{|(y,\eta)| \leqs n} |V_\fy u(y,\eta)| \, e^{r \left( |x|^{\frac1t} + |\xi|^{\frac1s} \right) } \, |V_\fy \fy(x-y, \xi - \eta)| \, \dd y \, \dd \eta \\
& \lesssim \int_{|(y,\eta)| \leqs n} e^{r_1 \left( |y|^{\frac1t} + |\eta|^{\frac1s} \right) } \, e^{r \left( |x|^{\frac1t} + |\xi|^{\frac1s} \right) } \, e^{ - \kappa r \left( |x-y|^{\frac1t} + |\xi-\eta|^{\frac1s} \right) } \, \dd y \, \dd \eta \\
& \lesssim \int_{|(y,\eta)| \leqs n}  e^{r_1 \left( |y|^{\frac1t} + |\eta|^{\frac1s} \right) + r \kappa \left( |y|^{\frac1t} + |\eta|^{\frac1s} \right) } \, \dd y \, \dd \eta \\
& \leqs C_{n,r,r_1}, \quad (x,\xi) \in \rr {2d}. 
\end{align*}
It follows that $u_n \in \Sigma_t^s(\rr d)$ for $n \in \no$. 

The fact that $u_n \to u$ in $( \Sigma_t^s )'(\rr d)$ as $n \to \infty$ is a consequence of 
\eqref{eq:moyal}, \eqref{eq:STFTGFstdistr}, \eqref{eq:STFTGFstfunc} and dominated convergence. 

We also need the estimate (cf. \cite[Eq.~(11.29)]{Grochenig1})
\begin{equation*}
|V_{\overline{\fy}} u_n (y,\eta)| \leqs (2 \pi)^{-\frac{d}{2}} |V_\varphi u| * |V_{\overline{\fy}} \fy| (y,\eta), \quad (y,\eta) \in \rr {2d}, 
\end{equation*}
which in view of \eqref{eq:STFTGFstdistr} and \eqref{eq:STFTGFstfunc}
gives the bound
\begin{equation}\label{eq:STFTupperbound3}
|V_{\overline{\fy}} u_n (y,\eta)| \lesssim e^{\kappa (1+r_1) \left( | y |^{\frac1t} + | \eta |^{\frac1s} \right) }, \quad ( y,\eta ) \in \rr {2d}, \quad n \in \no, 
\end{equation}
for some $r_1 > 0$, that holds uniformly over $n \in \no$. 

We are now in a position to assemble the arguments into a proof of formula 
\eqref{eq:opSTFT1} for $u \in ( \Sigma_t^s )'(\rr d)$ and $\psi \in \Sigma_t^s(\rr d)$. 
Set
\begin{equation}\label{eq:STFTequalitylimit1}
(\cK u, \psi) 
= \lim_{n \to \infty} (\cK u_n, \psi) 
= \lim_{n \to \infty} 
\int_{\rr {4d}} V_\Phi K(x,y,\xi,-\eta) \, \overline{V_\fy \psi (x,\xi)} \, V_{\overline \fy} u_n (y,\eta) \, \dd x \, \dd y \, \dd \xi \, \dd \eta.
\end{equation}

We have $V_{\overline \fy} u_n(y,\eta) \to V_{\overline \fy} u(y,\eta)$ as $n \to \infty$ for all $(y,\eta) \in \rr {2d}$. 
In order to show that the right hand side of \eqref{eq:STFTequalitylimit1} is well defined and \eqref{eq:opSTFT1} holds, 
it thus suffices by dominated convergence to show that the modulus of the integrand in \eqref{eq:STFTequalitylimit1} is bounded by an integrable function that does not depend on $n \in \no$.

Consider first the right hand side integral \eqref{eq:STFTequalitylimit1} over $(x,y,\xi,-\eta) \in \rr {4d} \setminus \Gamma_1$
where $\Gamma_1$ is defined by \eqref{eq:Gamma1} with $c > 1$ again chosen 
so that $\WF^{t,s} (K) \subseteq \Gamma_1$.
By Lemma \ref{lem:WFkernelaxes} we may use the estimates \eqref{eq:WFKcomplement}. 
Using \eqref{eq:STFTupperbound3} we obtain for any $r_2 > 0$
\begin{equation}\label{eq:seminormest1}
\begin{aligned}
& \int_{\rr {4d} \setminus \Gamma_1'} 
|V_\Phi K(x,y,\xi,-\eta)| \,  |V_\fy \psi (x,\xi)| \, |V_{\overline \fy} u_n (y,\eta)| \, \dd x \, \dd y \, \dd \xi \, \dd \eta \\
& \lesssim 
\int_{\rr {4d} \setminus \Gamma_1'} 
e^{ - r_2 \left( |(x,y)|^{\frac1t} + |(\xi,\eta)|^{\frac1s} \right) } \, 
|V_\fy \psi (x,\xi)| \, 
e^{\kappa (1+r_1) \left( | y |^{\frac1t} + | \eta |^{\frac1s} \right) } \, 
\dd x \, \dd y \, \dd \xi \, \dd \eta \\
& \lesssim 
\| \psi \|_0''
\int_{\rr {4d}}  e^{ ( \kappa (1+r_1) - r_2) \left( |(x,y)|^{\frac1t} + |(\xi,\eta)|^{\frac1s} \right) }  \, \dd x \, \dd y \, \dd \xi \, \dd \eta \\
& \lesssim \| \psi \|_0'' < \infty
\end{aligned}
\end{equation}
provided $r_2 > \kappa (1+r_1)$. 

Next we consider the right hand side integral \eqref{eq:STFTequalitylimit1} over $(x,y,\xi,-\eta) \in \Gamma_1$. 
Then we may by Lemma \ref{lem:WFkernelaxes} 
use \eqref{eq:Gamma1}. 
From \eqref{eq:STFTGFstdistr} and again \eqref{eq:STFTupperbound3} we obtain for some $r_2 > 0$
\begin{equation}\label{eq:seminormest2}
\begin{aligned}
& \int_{\Gamma_1'} 
|V_\Phi K(x,y,\xi,-\eta)| \,  |V_\fy \psi (x,\xi)| \, |V_{\overline \fy} u_n (y,\eta)| \, \dd x \, \dd y \, \dd \xi \, \dd \eta \\
& \lesssim 
\int_{\Gamma_1'} 
e^{ r_2 \left( |(x,y)|^{\frac1t} + |(\xi,\eta)|^{\frac1s} \right)} \, |V_\fy \psi (x,\xi)| \, e^{ \kappa (1+r_1) \left( |y|^{\frac1t} + |\eta|^{\frac1s} \right)} \, \dd x \, \dd y \, \dd \xi \, \dd \eta \\
& \leqs
\int_{\Gamma_1'} 
e^{ - \left( |(x,y)|^{\frac1t} + |(\xi,\eta)|^{\frac1s} \right)} \,  
e^{ \kappa (1 + r_2) \left( |x|^{\frac1t} + |y|^{\frac1t} + |\xi|^{\frac1s} + |\eta|^{\frac1s} \right) + \kappa (1+r_1) \left( |y|^{\frac1t} + |\eta|^{\frac1s} \right)}
\, |V_\fy \psi (x,\xi)| \, \dd x \, \dd y \, \dd \xi \, \dd \eta \\
& \leqs
\int_{\Gamma_1'} 
e^{ - \left( |(x,y)|^{\frac1t} + |(\xi,\eta)|^{\frac1s} \right)} \,  
e^{ \left( \kappa (1 +  r_2) (1+c) + \kappa (1+r_1) c \right) \left( |x|^{\frac1t} + |\xi|^{\frac1s} \right)}
\, |V_\fy \psi (x,\xi)| \, \dd x \, \dd y \, \dd \xi \, \dd \eta \\
& \lesssim 
\| \psi \|_{\kappa \left( (1 + r_2) (1+c) +  (1+r_1) c \right) }'' < \infty. 
\end{aligned}
\end{equation}
The estimates \eqref{eq:seminormest1} and \eqref{eq:seminormest2} prove our claim that  
the modulus of the integrand in right hand side of \eqref{eq:STFTequalitylimit1} is bounded by an $L^1(\rr {4d})$ function uniformly over $N \in \no$. 
Thus \eqref{eq:STFTequalitylimit1} extends the domain of $\cK$ from $\Sigma_t^s(\rr d)$ to $\left( \Sigma_t^s \right)'(\rr d)$. 
We have shown claim (iii). 

From \eqref{eq:seminormest1} and \eqref{eq:seminormest2} we also see that $\cK u$ extended to the domain $u \in \left( \Sigma_t^s \right)'(\rr d)$ satisfies 
$\cK u \in \left( \Sigma_t^s \right)' (\rr d)$. To prove claim (ii) it remains to show the continuity of the extension \eqref{eq:STFTequalitylimit1} on $\left( \Sigma_t^s \right)'(\rr d)$. 
The uniqueness of the extension is a consequence of the continuity. 

Let $(u_n)_{n = 1}^\infty \subseteq \left( \Sigma_t^s \right)'(\rr d)$ be a sequence such that $u_n \to 0$ in $\left( \Sigma_t^s \right)'(\rr d)$ as $n \to \infty$. 
Then $V_{\overline \fy} u_n(y,\eta) \to 0$ as $n \to \infty$ for all $(y,\eta) \in \rr {2d}$. 
By the Banach--Steinhaus theorem \cite[Theorem~V.7]{Reed1}, $(u_n)_{n = 1}^\infty$ is equicontinuous. 
This means that there exists $r > 0$ such that 
\begin{equation*}
|(u_n, \psi)| 
\lesssim \| \psi \|_r ''
= \sup_{(x,\xi) \in \rr {2d}} e^{ r \left( |x|^{\frac1t} + |\xi|^{\frac1s} \right)}  |V_\fy \psi (x,\xi)|, \quad \psi \in \Sigma_t^s(\rr d), \quad n \in \no. 
\end{equation*}

Hence
\begin{align*}
|V_{\overline \fy} u_n(y,\eta)| 
& = (2 \pi)^{- \frac{d}{2}} |(u_n, \Pi(y,\eta) \overline \fy )| 
\lesssim \sup_{(x,\xi) \in \rr {2d}} e^{ r \left( |x|^{\frac1t} + |\xi|^{\frac1s} \right)} |V_\fy (\Pi(y,\eta) \overline \fy) (x,\xi)| \\
& = \sup_{(x,\xi) \in \rr {2d}} e^{ r \left( |x|^{\frac1t} + |\xi|^{\frac1s} \right)} |V_\fy \overline \fy (x-y,\xi-\eta) | \\
& \lesssim \sup_{(x,\xi) \in \rr {2d}} e^{ r \left( |x|^{\frac1t} + |\xi|^{\frac1s} \right)} 
e^{ - r \kappa \left( |x-y|^{\frac1t} + |\xi-\eta|^{\frac1s} \right)}
\lesssim e^{ r \kappa \left( |y|^{\frac1t} + |\eta|^{\frac1s} \right)} , \quad (y,\eta) \in \rr {2d},
\end{align*}
uniformly for all $n \in \no$. 
From \eqref{eq:opSTFT1}, the estimates \eqref{eq:seminormest1}, \eqref{eq:seminormest2}, and dominated convergence it follows 
that $(\cK u_n, \psi) \to 0$ as $n \to \infty$ for all $\psi \in \Sigma_t^s(\rr d)$, that is $\cK u_n \to 0$ in $\left( \Sigma_t^s \right)'(\rr d)$. 
This finally proves claim (ii). 
\end{proof}

Now we start to prepare for the main result Theorem \ref{thm:WFGSpropagation}. 
We will use the relation mapping between a subset $A \subseteq X \times Y$ of the Cartesian product of two sets $X$, $Y$, and a subset $B \subseteq Y$, 
\begin{equation*}
A \circ B = \{ x \in X: \, \exists y \in B: \, (x,y) \in A \} \subseteq X.  
\end{equation*}

When $X = Y = \rr {2d}$ we use the convention
\begin{equation*}
A' \circ B  = \{ (x,\xi) \in \rr {2d}: \,  \exists (y,\eta) \in B: \, (x,y,\xi,-\eta) \in A \}. 
\end{equation*}
Note that we use \eqref{eq:reflection}, and there is a swap of the second and third variables. 

If we denote by 
\begin{align*}
p_{1,3}(x,y,\xi,\eta) & = (x,\xi), \\
p_{2,-4}(x,y,\xi,\eta) & = (y,-\eta), \quad x,y,\xi, \eta \in \rr d, 
\end{align*}
the projections $\rr {4d} \rightarrow \rr {2d}$ onto the first and the third $\rr d$ coordinate, 
and onto the second and the fourth $\rr d$ coordinate with a change of sign in the latter, respectively, then we may write
\begin{equation}\label{eq:relationproj}
\WF^{t,s} (K)' \circ \WF^{t,s} (u) 
= p_{1,3} \left( \WF^{t,s} (K) \cap p_{2,-4}^{-1} \WF^{t,s} (u) \right). 
\end{equation}

\begin{lem}\label{lem:sconicclosed}
If $t,s > 0$, $t + s > 1$, $K \in \left( \Sigma_t^s \right)'(\rr {2d})$, \eqref{eq:WFKjempty} holds and $u \in \left( \Sigma_t^s \right)'(\rr d)$
then 
\begin{equation*}
\WF^{t,s} (K)' \circ \WF^{t,s} (u) \subseteq T^* \rr d \setminus 0
\end{equation*}
is $\frac{s}{t}$-conic and closed in $T^* \rr d \setminus 0$. 
\end{lem}

\begin{proof}
Let $(x,\xi) \in \WF^{t,s} (K)' \circ \WF^{t,s} (u)$. 
Then there exists $(y,\eta) \in \WF^{t,s} (u)$ such that $(x,y,\xi,-\eta) \in \WF^{t,s} (K)$. 

Let $\lambda > 0$.
Since $\WF^{t,s} (K)$ and $\WF^{t,s} (u)$ are $\frac{s}{t}$-conic we have 
$( \lambda^t x, \lambda^t y, \lambda^s \xi,- \lambda^s \eta) \in \WF^{t,s} (K)$
and $(\lambda^t y, \lambda^s \eta) \in \WF^{t,s} (u)$. 
It follows that $(\lambda^t x, \lambda^s \xi) \in \WF^{t,s} (K)' \circ \WF^{t,s} (u)$ 
which shows that $\WF^{t,s} (K)' \circ \WF^{t,s} (u)$ is $\frac{s}{t}$-conic. 

Next we assume that $(x_n,\xi_n) \in \WF^{t,s} (K)' \circ \WF^{t,s} (u)$ for $n \in \no$
and $(x_n, \xi_n) \to (x,\xi) \neq 0$ as $n \to +\infty$. 
For each $n \in \no$ there exists $(y_n,\eta_n) \in \WF^{t,s} (u)$ such that 
$(x_n,y_n,\xi_n,-\eta_n) \in \WF^{t,s} (K)$. 

Since the sequence $\{ (x_n, \xi_n)_n \} \subseteq T^* \rr d$ is bounded it follows from Lemma \ref{lem:WFkernelaxes}
that also the sequence $\{ (y_n, \eta_n)_n \} \subseteq T^* \rr d$ is bounded. 
Passing to a subsequence (without change of notation) we get convergence
\begin{equation*}
\lim_{n \to +\infty} (x_n,y_n,\xi_n,-\eta_n) 
= (x,y,\xi,-\eta) \in \rr {4d} \setminus 0. 
\end{equation*}
Here $(x,y,\xi,-\eta) \in \WF^{t,s} (K)$ since $\WF^{t,s} (K) \subseteq T^* \rr {2d} \setminus 0$ is closed, and $(y,\eta) \neq 0$ due to the assumption $\WF_1^{t,s} (K) = \emptyset$.
Since $\WF^{t,s} (u)\subseteq T^* \rr {d} \setminus 0$ is closed we have $(y,\eta) \in \WF^{t,s} (u)$.
We have proved that $(x,\xi) \in \WF^{t,s} (K)' \circ \WF^{t,s} (u)$ which shows that $\WF^{t,s} (K)' \circ \WF^{t,s} (u)$ is closed in $T^* \rr d \setminus 0$. 
\end{proof}

Let $s > 0$, let $G \subseteq T^* \rr d \setminus 0$ be a closed $s$-conic subset, and let $\ep > 0$. 
In the next result we use the notation
\begin{equation}\label{eq:GammaG}
\Gamma_{G,\ep} = \{ z \in T^* \rr d \setminus 0: \, \inf_{w \in G \cap \sr {2d-1}} | p_{1,s}(z) - w| < \ep \}.
\end{equation}
This generalizes Definition \ref{def:scone1} since $\Gamma_{G,\ep} =  \Gamma_{(x_0,\xi_0),\ep}$
if $G = \{ (\lambda x_0, \lambda^s \xi_0) \in T^* \rr d \setminus 0: \lambda > 0 \}$ and $(x_0,\xi_0) \in \sr {2d-1}$. 
Note that $\Gamma_{G,\ep}$ is an open $s$-conic set, and $G \subseteq \Gamma_{G,\ep}$.

\begin{lem}\label{lem:separationsconic}
Suppose $G_j \subseteq T^* \rr {2d} \setminus 0$ is closed $s$-conic for $j=1,2$, 
suppose $G_3 \subseteq T^* \rr d \setminus 0$ is closed $s$-conic, 
and suppose 
\begin{equation*}
G_1 \cap G_2 \cap p_{2,-4}^{-1} \left( G_3 \cup \{ 0 \} \right) \setminus 0 = \emptyset. 
\end{equation*}
Define $\Gamma_{j,\ep} = \Gamma_{G_j,\ep}$ for $\ep > 0$ and $j=1,2,3$. 
Then for some $\ep > 0$ we have 
\begin{equation*}
\Gamma_{1,\ep} \cap \Gamma_{2,\ep} \cap p_{2,-4}^{-1}  \left( \Gamma_{3,\ep} \cup \{ 0 \} \right) \setminus 0 = \emptyset. 
\end{equation*}
\end{lem}

\begin{proof}
Note that $p_{2,-4}^{-1} \left( G_3 \cup \{ 0 \} \right) \setminus 0$ is closed $s$-conic in $T^* \rr {2d} \setminus 0$, and 
$p_{2,-4}^{-1}  \left( \Gamma_{3,\ep} \cup \{ 0 \} \right) \setminus 0$ is $s$-conic in $T^* \rr {2d} \setminus 0$ for any $\ep > 0$. 

Suppose that for each $n \in \no$ we have 
\begin{equation*}
X_n = (x_n, y_n, \xi_n, \eta_n) \in \Gamma_{1,\frac1n} \cap \Gamma_{2,\frac1n} \cap p_{2,-4}^{-1}  \left( \Gamma_{3,\frac1n} \cup \{ 0 \} \right) \setminus 0. 
\end{equation*}
Since $\Gamma_{j,\frac1n}$ for $j=1,2$, as well as $p_{2,-4}^{-1}  \left( \Gamma_{3,\frac1n} \cup \{ 0 \} \right) \setminus 0$,
are $s$-conic in $T^* \rr {2d} \setminus 0$,
we may assume that $| X_n | = 1$ for all $n \in \no$. 
Passing to a subsequence (without change of notation) we get $X_n \to X = (x, y, \xi, \eta) \in \sr {4d-1}$ as $n \to + \infty$. 

For each $n \in \no$ and $j = 1,2$ there exists $Y_{j,n} \in G_j \cap \sr {4d-1}$ such that $|X_n - Y_{j,n} | < \frac1n$. 
Thus $| X - Y_{j,n} | \leqs | X - X_n | + | X_n - Y_{j,n} | \to 0$ as $n \to + \infty$. 
Since $G_j$ is closed for $j=1,2$, 
it follows that $X \in G_1 \cap G_2$. 

Suppose $(y,\eta) = 0$. Then $X \in p_{2,-4}^{-1} \left( G_3 \cup \{ 0 \} \right) \setminus 0$
and thus 
\begin{equation*}
X \in G_1 \cap G_2 \cap p_{2,-4}^{-1} \left( G_3 \cup \{ 0 \} \right) \setminus 0
\end{equation*}
which contradicts the assumption.
Hence $(y,\eta) \neq 0$ must hold, and therefore $(y_n,-\eta_n) \in \Gamma_{3,\frac1n}$ if $n \geqs N$ for $N > 0$ sufficiently large. 

For each $n \geqs N$ there exists $Y_n \in G_3 \cap \sr {2d-1}$ such that $| p_{1,s}(y_n,-\eta_n) - Y_n | < \frac1n$. 
This gives $| p_{1,s}(y,-\eta) - Y_n | \leqs | p_{1,s}(y,-\eta) - p_{1,s}(y_n,-\eta_n) | + | p_{1,s}(y_n,-\eta_n) - Y_n | \to 0$
as $n \to + \infty$, taking into account the fact that $p_{1,s}$ is continuous. 
Since $G_3$ is closed it follows that $p_{1,s}(y,-\eta) \in G_3$. 
This implies $(y,-\eta) \in G_3$ using the fact that $G_3$ is $s$-conic. 
We arrive at the conclusion $X \in p_{2,-4}^{-1} \left( G_3 \cup \{ 0 \} \right) \setminus 0$
which again contradicts the assumption. 

We may conclude that for some $n \in \no$ we must have  
\begin{equation*}
\Gamma_{1,\frac1n} \cap \Gamma_{2,\frac1n} \cap p_{2,-4}^{-1}  \left( \Gamma_{3,\frac1n} \cup \{ 0 \} \right) \setminus 0 = \emptyset.
\end{equation*}
\end{proof}

Finally we may state and prove our main result on propagation of singularities.

\begin{thm}\label{thm:WFGSpropagation}
Let $t,s > 0$ satisfy $t + s > 1$, and let $\cK: \Sigma_t^s(\rr d) \to \left( \Sigma_t^s \right)' (\rr d)$ be the continuous linear operator \eqref{eq:kernelop}
defined by the Schwartz kernel $K \in \left( \Sigma_t^s \right)' (\rr {2d})$, and suppose that \eqref{eq:WFKjempty} holds. 
Then $\cK$ is continuous on $\Sigma_t^s(\rr d)$, extends uniquely to a continuous operator on $\left( \Sigma_t^s \right)' (\rr d)$,  
and for $u \in \left( \Sigma_t^s \right)' (\rr d)$
we have
\begin{equation*}
\WF^{t,s} (\cK u) \subseteq \WF^{t,s} (K)' \circ \WF^{t,s} (u).  
\end{equation*}
\end{thm}

\begin{proof}
By Proposition \ref{prop:opSTFTformula} $\cK: \Sigma_t^s (\rr d) \to \Sigma_t^s(\rr d)$ is continuous and extends
uniquely to a continuous linear operator $\cK: \left( \Sigma_t^s \right)' (\rr d) \to \left( \Sigma_t^s \right)' (\rr d)$. 

Let $\fy \in \Sigma_t^s(\rr d)$ satisfy $\| \fy \|_{L^2} = 1$ and set $\Phi = \fy \otimes \fy \in \Sigma_t^s(\rr {2d})$. 
Proposition \ref{prop:opSTFTformula}, \eqref{eq:opSTFT1} and \eqref{eq:STFTKu} give
for $u \in \left( \Sigma_t^s \right)' (\rr d)$ and $(x,\xi) \in T^* \rr d$ and $\lambda > 0$
\begin{equation}\label{eq:opSTFTest1}
| V_\fy(\cK u) ( \lambda^t x, \lambda^s \xi) | 
\lesssim \int_{\rr {4d}} | V_\Phi K (y,z,\eta,-\theta) | \, |V_\fy \fy ( \lambda^t x-y, \lambda^s \xi-\eta) | \, | V_{\overline \fy} u(z,\theta) | \, \dd y \, \dd z \, \dd \eta \, \dd \theta. 
\end{equation}

We may assume that $\WF^{t,s} (K)' \circ \WF^{t,s} (u) \neq T^* \rr d \setminus 0$ since the conclusion is trivial otherwise. 
Suppose $z_0 = (x_0,\xi_0) \in T^* \rr d \setminus 0$ and
\begin{equation}\label{eq:notWFs1}
z_0 \notin \WF^{t,s} (K)' \circ \WF^{t,s} (u).
\end{equation}
To prove the theorem we show $z_0\notin \WF^{t,s} (\cK u)$. 

By Lemma \ref{lem:sconicclosed} the set $\WF^{t,s} (K)' \circ \WF^{t,s} (u)$ is $\frac{s}{t}$-conic and closed. 
Thus we may assume that $z_0 \in \sr {2d-1}$.
Moreover, with $\wt \Gamma_{z_0,2 \ep} = \wt \Gamma_{\frac{s}{t}, z_0, 2 \ep}$, there exists $\ep > 0$ such that 
\begin{equation*}
\overline{\wt \Gamma}_{z_0,2 \ep} \cap \left( \WF^{t,s} (K)' \circ \WF^{t,s} (u) \right)= \emptyset. 
\end{equation*}
Here $\overline{\wt \Gamma}_{z_0,2 \ep}$ denotes the closure of $\wt \Gamma_{z_0,2 \ep}$ in $T^* \rr d \setminus 0$. 
Using \eqref{eq:relationproj} we may write this as
\begin{equation*}
\overline{\wt \Gamma}_{z_0,2 \ep} \cap p_{1,3} \left( \WF^{t,s} (K) \cap p_{2,-4}^{-1} \WF^{t,s} (u) \right)= \emptyset
\end{equation*}
or equivalently 
\begin{equation*}
p_{1,3}^{-1} \overline{\wt \Gamma}_{z_0,2 \ep} \cap \WF^{t,s} (K) \cap p_{2,-4}^{-1} \WF^{t,s} (u) = \emptyset. 
\end{equation*}

Due to assumption \eqref{eq:WFKjempty} we may strengthen this into 
\begin{equation*}
p_{1,3}^{-1} \, (\overline{\wt \Gamma}_{z_0,2 \ep} \cup \{ 0 \} ) \setminus 0 \cap \WF^{t,s} (K) \cap p_{2,-4}^{-1} \, (\WF^{t,s} (u) \cup \{ 0 \} ) \setminus 0 = \emptyset.  
\end{equation*}
Note that $p_{1,3}^{-1} \, (\overline{\wt \Gamma}_{z_0,2 \ep} \cup \{ 0 \} ) \setminus 0$,  
$\WF^{t,s} (K)$,
and $p_{2,-4}^{-1} \, (\WF^{t,s} (u) \cup \{ 0 \} ) \setminus 0$ are all closed and $\frac{s}{t}$-conic
subsets of $T^* \rr {2d} \setminus 0$. 

Now Lemma \ref{lem:separationsconic} gives the following conclusion. 
There exists open $\frac{s}{t}$-conic subsets $\Gamma_1 \subseteq T^* \rr {2d} \setminus 0$ 
and $\Gamma_2 \subseteq T^* \rr d \setminus 0$
such that  
\begin{equation*}
\WF^{t,s} (K) \subseteq \Gamma_1, \quad \WF^{t,s} (u) \subseteq \Gamma_2
\end{equation*}
and 
\begin{equation}\label{eq:emptyintersection}
p_{1,3}^{-1} \overline{\wt \Gamma}_{z_0,2 \ep}
\cap \Gamma_1 \cap p_{2,-4}^{-1} \Gamma_2 = \emptyset. 
\end{equation}
By intersecting $\Gamma_1$ with the set $\Gamma_1$ defined in \eqref{eq:Gamma1}, 
we may by Lemma \ref{lem:WFkernelaxes} assume that \eqref{eq:Gamma1} holds true. 

Let $r > 0$. 
We will now start to estimate the integral \eqref{eq:opSTFTest1} when $(x,\xi) \in (x_0, \xi_0) + \rB_\ep$ for some $0 < \ep \leqs \frac12$
and $\lambda \geqs 1$. 

We split the domain $\rr {4d}$ of the integral \eqref{eq:opSTFTest1} into three pieces. 
Set $\kappa = \max (\kappa(t^{-1}), \kappa(s^{-1}) )$ and 
\begin{align*}
\delta & = \inf_{(x,\xi) \in (x_0, \xi_0) + \rB_\ep} |x|^{\frac1t} + |\xi|^{\frac1s} > 0, \\
\Delta & = \sup_{(x,\xi) \in (x_0, \xi_0) + \rB_\ep} |x|^{\frac1t} + |\xi|^{\frac1s} < + \infty. 
\end{align*}
First we integrate over $\rr {4d} \setminus \Gamma_1'$ where we may use 
\eqref{eq:WFKcomplement}. 
Combined with \eqref{eq:STFTGFstdistr} and \eqref{eq:STFTGFstfunc} this gives if $(x,\xi) \in (x_0, \xi_0) + \rB_\ep$ 
for some $r_1 > 0$ and any $r_2 > 0$
\begin{equation}\label{eq:opSTFTsubest1}
\begin{aligned}
& \int_{\rr {4d} \setminus \Gamma_1'} | V_\Phi K (y,z,\eta,-\theta) | \, |V_\fy \fy ( \lambda^t x-y, \lambda^s \xi-\eta) | \, | V_{\overline \fy} u(z,\theta) | \, \dd y \, \dd z \, \dd \eta \, \dd \theta \\
& \lesssim \int_{\rr {4d} \setminus \Gamma_1'} 
e^{- r_2 \left( |(y,z)|^{\frac1t} + |( \eta,\theta )|^{\frac1s} \right) }  
\, e^{- r \delta^{-1} \kappa \left( |\lambda^t x-y|^{\frac1t} + |\lambda^s \xi-\eta|^{\frac1s} \right) } 
\, e^{r_1 \left( |z|^{\frac1t} + |\theta|^{\frac1s} \right) } 
\, \dd y \, \dd z \, \dd \eta \, \dd \theta \\
& \leqs
e^{- r \delta^{-1} \lambda \left( |x|^{\frac1t} + |\xi|^{\frac1s} \right) }
\int_{\rr {4d} \setminus \Gamma_1'} 
e^{- r_2 \left( |(y,z)|^{\frac1t} + |( \eta,\theta )|^{\frac1s} \right)   
+ r \delta^{-1}  \kappa \left( |y|^{\frac1t} + |\eta|^{\frac1s} \right) 
+ r_1 \left( |z|^{\frac1t} + |\theta|^{\frac1s} \right) } 
\, \dd y \, \dd z \, \dd \eta \, \dd \theta \\
& \leqs
e^{- r \lambda }
\int_{\rr {4d}} 
e^{ (r_1 + r \delta^{-1} \kappa - r_2) \left( |(y,z)|^{\frac1t} + |( \eta,\theta )|^{\frac1s} \right) } 
\, \dd y \, \dd z \, \dd \eta \, \dd \theta \\
& \lesssim 
e^{- r \lambda}
\end{aligned}
\end{equation}
provided $r_2 > r_1 + r \delta^{-1} \kappa$. 

It remains to estimate the integral \eqref{eq:opSTFTest1} over $(y,z,\eta, - \theta) \in \Gamma_1$ where
we may use \eqref{eq:Gamma1}. 
By \eqref{eq:emptyintersection} we have 
\begin{equation}\label{eq:Gamma1union}
\Gamma_1 \subseteq \Omega_0 \cup \Omega_2
\end{equation}
where 
\begin{equation*}
\Omega_0 = \Gamma_1 \setminus p_{1,3}^{-1} \overline{\wt \Gamma}_{z_0,2 \ep}, \quad
\Omega_2 = \Gamma_1 \setminus p_{2,-4}^{-1} \Gamma_2.
\end{equation*}

First we estimate the integral over $(y,z,\eta, - \theta) \in \Omega_2$. 
Then $(z,\theta) \in \rr {2d} \setminus \Gamma_2$ which is a closed $\frac{s}{t}$-conic set. 
By $\WF^{t,s} (u) \subseteq \Gamma_2$, the compactness of $\sr {2d-1} \setminus \Gamma_2$ 
and \eqref{eq:WFgs2} we obtain the estimates
\begin{equation*}
|V_\fy u (z,\theta)| 
\lesssim e^{- r_2 \left( |z|^{\frac1t} + |\theta|^{\frac1s} \right) }, \quad (z,\theta) \in \rr {2d} \setminus \Gamma_2, \quad \forall r_2 > 0.
\end{equation*}

Together with \eqref{eq:Gamma1} and \eqref{eq:STFTGFstdistr} this gives if $(x,\xi) \in (x_0, \xi_0) + \rB_\ep$ 
for some $r_1 > 0$
\begin{equation}\label{eq:opSTFTsubest2}
\begin{aligned}
& \int_{\Omega_2'} | V_\Phi K (y,z,\eta,-\theta) | \, |V_\fy \fy ( \lambda^t x-y, \lambda^s \xi-\eta) | \, | V_{\overline \fy} u(z,\theta) | \, \dd y \, \dd z \, \dd \eta \, \dd \theta \\
& \lesssim 
\int_{\Omega_2'} 
e^{r_1 \left( |(y,z)|^{\frac1t} + |( \eta,\theta )|^{\frac1s} \right) 
- r \delta^{-1} \kappa \left( |\lambda^t x-y|^{\frac1t} + |\lambda^s \xi-\eta|^{\frac1s} \right) } 
\, | V_{\overline \fy} u(z,\theta) | \, \dd y \, \dd z \, \dd \eta \, \dd \theta \\
& \leqs
e^{- r \delta^{-1} \lambda \left( |x|^{\frac1t} + |\xi|^{\frac1s} \right) }
\int_{\Omega_2'} 
e^{r_1 \left( |(y,z)|^{\frac1t} + |( \eta,\theta )|^{\frac1s} \right) 
+ r \delta^{-1}  \kappa \left( |y|^{\frac1t} + |\eta|^{\frac1s} \right) } 
\, | V_{\overline \fy} u(z,\theta) | \, \dd y \, \dd z \, \dd \eta \, \dd \theta \\
& \leqs
e^{- r \lambda }
\int_{\Omega_2'} 
e^{- \left( |(y,z)|^{\frac1t} + |( \eta,\theta )|^{\frac1s} \right) }
\, e^{ (1 + r_1) \kappa \left( |y|^{\frac1t} + |z|^{\frac1t} + |\eta|^{\frac1s} + |\theta|^{\frac1s} \right) 
+ r \delta^{-1}  \kappa \left( |y|^{\frac1t} + |\eta|^{\frac1s} \right) } 
\, | V_{\overline \fy} u(z,\theta) | \, \dd y \, \dd z \, \dd \eta \, \dd \theta \\
& \leqs
e^{- r \lambda }
\sup_{(z,\theta) \in \rr {2d} \setminus \Gamma_2}
e^{  \kappa \left( (1 + r_1) (1+c) + r \delta^{-1} c \right) \left( |z|^{\frac1t} + |\theta|^{\frac1s} \right) } 
\, | V_{\overline \fy} u(z,\theta) |
\int_{\rr {4d}} 
e^{- \left( |(y,z)|^{\frac1t} + |( \eta,\theta )|^{\frac1s} \right) }
\, \dd y \, \dd z \, \dd \eta \, \dd \theta \\
& \lesssim 
e^{- r \lambda }. 
\end{aligned}
\end{equation}

Finally we need to estimate the integral over $(y,z,\eta, - \theta) \in \Omega_0$. 
Then $(y,\eta) \in \rr {2d} \setminus \overline{\wt \Gamma}_{z_0, 2 \ep}$. 
Hence
\begin{equation*}
\left| z_0 - \left( \lambda^{-t} y, \lambda^{-s} \eta \right) \right| \geqs 2 \ep \quad \forall \lambda > 0 \quad \forall (y,\eta) \in \rr {2d} \setminus \overline{\wt \Gamma}_{z_0, 2 \ep}  
\end{equation*}
and we have for $(x,\xi) \in z_0 + \rB_\ep$
\begin{equation*}
\left| (x,\xi) - \left( \lambda^{-t} y, \lambda^{-s} \eta \right) \right| \geqs \ep \quad \forall \lambda > 0 \quad \forall (y,\eta) \in \rr {2d} \setminus \overline{\wt \Gamma}_{z_0, 2 \ep}. 
\end{equation*}

It follows that there exists $\alpha > 0$ such that 
for $\lambda \geqs 1$, $(x,\xi) \in z_0 + \rB_\ep$ and 
$(y,\eta) \in \rr {2d} \setminus \overline{\wt \Gamma}_{z_0, 2 \ep}$ we have
\begin{equation*}
|\lambda^t x-y|^{\frac1t} + |\lambda^s \xi-\eta|^{\frac1s}
= \lambda \left( | x -  \lambda^{-t} y |^{\frac1t} +  | \xi - \lambda^{-s} \eta |^{\frac1s} \right) 
\geqs \lambda \alpha. 
\end{equation*}

Together with \eqref{eq:Gamma1} and \eqref{eq:STFTGFstdistr} this gives if $(x,\xi) \in (x_0, \xi_0) + \rB_\ep$ 
for some $r_1, r_2 > 0$ and any $r_3, r_4 > 0$
\begin{equation}\label{eq:opSTFTsubest3}
\begin{aligned}
& \int_{\Omega_0'} | V_\Phi K (y,z,\eta,-\theta) | \, |V_\fy \fy ( \lambda^t  x-y, \lambda^s \xi-\eta) | \, | V_{\overline \fy} u(z,\theta) | \, \dd y \, \dd z \, \dd \eta \, \dd \theta \\
& \lesssim \int_{\Omega_0'} 
e^{r_1 \left( |(y,z)|^{\frac1t} + |( \eta,\theta )|^{\frac1s} \right) + r_2 \left( |z|^{\frac1t} + |\theta|^{\frac1s} \right)}
\, e^{- \left( \frac{r_3}{\alpha} + r_4 \kappa \right) \left( |\lambda^t x-y|^{\frac1t} + |\lambda^s \xi-\eta|^{\frac1s} \right) } 
\, \dd y \, \dd z \, \dd \eta \, \dd \theta \\
& \leqs
e^{- r_3 \lambda} 
\int_{\Omega_0'} 
e^{-\left( |(y,z)|^{\frac1t} + |( \eta,\theta )|^{\frac1s} \right) + (1 + r_1 + r_2) \kappa \left( |y|^{\frac1t} + |z|^{\frac1t} + |\eta|^{\frac1s} + |\theta|^{\frac1s}\right)}
\, e^{- r_4 \kappa \left( |\lambda^t x-y|^{\frac1t} + |\lambda^s \xi-\eta|^{\frac1s} \right) } 
\, \dd y \, \dd z \, \dd \eta \, \dd \theta \\
& \leqs
e^{- r_3 \lambda + r_4 \kappa \lambda \left( |x|^{\frac1t} + |\xi|^{\frac1t}\right) }
\int_{\Omega_0'} 
e^{-\left( |(y,z)|^{\frac1t} + |( \eta,\theta )|^{\frac1s} \right) + (1 + r_1 + r_2) (1+c)\kappa \left( |y|^{\frac1t} + |\eta|^{\frac1s} \right)}
\, e^{- r_4 \left( |y|^{\frac1t} + |\eta|^{\frac1s} \right) } 
\, \dd y \, \dd z \, \dd \eta \, \dd \theta \\
& \leqs
e^{- \lambda (r_3 - r_4 \kappa \Delta ) }
\int_{\Omega_0'} 
e^{-\left( |(y,z)|^{\frac1t} + |( \eta,\theta )|^{\frac1s} \right) + \left( (1 + r_1 + r_2) (1+c) \kappa - r_4 \right)\left( |y|^{\frac1t} + |\eta|^{\frac1s} \right)}
\, \dd y \, \dd z \, \dd \eta \, \dd \theta \\
& \lesssim 
e^{- r \lambda }
\end{aligned}
\end{equation}
if we first pick $r_4 \geqs (1 + r_1 + r_2) (1+c)\kappa$ and then $r_3 \geqs r + r_4 \kappa \Delta$. 

Combining \eqref{eq:opSTFTsubest1}, \eqref{eq:opSTFTsubest2} and \eqref{eq:opSTFTsubest3}
and taking into account \eqref{eq:Gamma1union}, we have by \eqref{eq:opSTFTest1} shown
\begin{equation*}
\sup_{(x,\xi) \in (x_0, \xi_0) + \rB_\ep, \ \lambda > 0} e^{r \lambda} | V_\fy(\cK u) ( \lambda^t x, \lambda^s \xi) | 
< + \infty \quad \forall r > 0
\end{equation*}
which finally proves the claim $z_0 \notin \WF^{t,s} (\cK u)$. 
\end{proof}

\section{The $t,s$-Gelfand--Shilov wave front set of oscillatory functions}\label{sec:chirp}

An important reason for the introduction of the $t,s$-Gelfand--Shilov anisotropic wave front set is that it describes accurately the phase space singularities of oscillatory functions known generically as chirp signals. 

Let $\fy: \rr d \to \ro$ be a real polynomial of order $m \geqs 2$
\begin{equation}\label{eq:phasefunction}
\fy (x) = \fy_m(x) + p(x)
\end{equation}
where 
\begin{equation}\label{eq:polynomial1}
p (x) = \sum_{0 \leqs |\alpha| < m} c_\alpha x^\alpha, \quad c_\alpha \in \ro, 
\end{equation}
and 
\begin{equation}\label{eq:principalpart}
\fy_m(x) = \sum_{|\alpha| = m} c_\alpha x^\alpha, \quad c_\alpha \in \ro, \quad \exists \alpha \in \nn d: \  |\alpha| = m, \ c_\alpha \in \ro \setminus 0, 
\end{equation}
is the principal part. 

In \cite{Rodino3} we investigate the $t,s$-Gelfand--Shilov wave front set of chirp functions defined on $\ro$.
Here we generalize this into the domain $\rr d$. 
Thus we study chirp functions of the form 
\begin{equation}\label{eq:chirpdef}
u(x) = e^{i \fy(x)}, \quad x \in \rr d. 
\end{equation}

First we note that for any $\lambda > 0$, any $t > 0$ and any $1 \leqs j \leqs d$ we have 
\begin{equation}\label{eq:phasederivative1}
\lambda^{- t m} \partial_j \left( \fy(\lambda^t y) \right)
= \partial_j \fy_m (y) + \lambda^{t(1-m)} \partial_j p(\lambda^t y) 
\end{equation}
and if $|y| \leqs R$ and $\lambda \geqs 1$ then 
\begin{equation}\label{eq:polderivative1}
\lambda^{t(1-m)} | \partial_j p(\lambda^t y) |
= \left| \sum_{0 \leqs |\alpha| \leqs m-1} \alpha_j c_\alpha y^{\alpha-e_j} \lambda^{t (|\alpha| - m)} \right|
\leqs C_R \lambda^{-t}. 
\end{equation}

The following result generalizes \cite[Theorem~4.2~(i)]{Rodino3} and shows that only the principal part $\fy_m(x)$ of $\fy$ 
is recorded in $\WF^{t,t(m-1)} (u)$, and the wave front set is contained in the $(m-1)$-conic set in phase space which is  
the graph of its gradient, that is $0 \neq x \mapsto ( x , \nabla \fy_m (x))$. 
The gradient of the phase function is known as the instantaneous frequency \cite{Boggiatto1}.

\begin{thm}\label{thm:chirpWFGS}
If $m \geqs 2$, $\fy$ is a real polynomial defined by \eqref{eq:phasefunction}, \eqref{eq:polynomial1}, \eqref{eq:principalpart},
$u$ is defined by \eqref{eq:chirpdef}, 
and $t > \frac1{m-1}$ then 
\begin{equation}\label{eq:chirpconclusion1}
\WF^{t,t(m-1)} (u) \subseteq \{ (x, \nabla \fy_m (x) ) \in \rr {2d}: \ x \neq 0 \}. 
\end{equation}
If $d = 1$ and $\fy$ is even or odd then 
\begin{equation}\label{eq:chirpconclusion2}
\WF^{t,t(m-1)} (u)
= \{ (x, \fy_m'(x) ) \in \rr 2: \ x \neq 0 \}. 
\end{equation}
\end{thm}

\begin{proof}
Set $s = t(m-1) > 1$. 
This implies that there are compactly supported Gevrey functions \cite{Rodino1} of order $s$ in the space $\Sigma_t^s(\rr d)$
which is a crucial ingredient in the proof. 
Set 
\begin{align*}
W 
& = \{ (x, \nabla \fy_m (x) ) \in \rr {2d}: \ x \in \rr d \setminus 0 \} \subseteq T^* \rr d \setminus 0.
\end{align*}
Then $W$ is an $(m-1)$-conic subset in $T^* \rr d \setminus 0$.

Suppose $(x_0, \xi_0) \in \rr {2d} \setminus 0$ and $(x_0, \xi_0) \notin W$. 
Then there exists $1 \leqs j \leqs d$ such that $\xi_{0,j} \neq \partial_j \fy_m (x_0)$. 
Thus there exist an open set $U$ such that $(x_0,\xi_0) \in U$, and $0 < \ep \leqs 1$, $\delta > 0$,
such that
\begin{equation*}
(x,\xi) \in U,  \quad |x-y| \leqs \delta \sqrt{2} \quad \Longrightarrow \quad  |\xi_j - \partial_j \fy_m (x) | \geqs 2 \ep, \quad | \partial_j (\fy_m (x) - \fy_m(y) )| \leqs \frac{\ep}{2}.
\end{equation*}
By \eqref{eq:polderivative1} we have 
\begin{equation*}
\lambda^{t(1-m)} | \partial_j p (\lambda^t y) |
\leqs \frac{\ep}{2}
\end{equation*}
if $(x,\xi) \in U$, $|x-y| \leqs \delta \sqrt{2}$ and $\lambda \geqs L$ where $L \geqs 1$ is sufficiently large. 

Using \eqref{eq:phasederivative1} we obtain if $(x,\xi) \in U$, $|x-y| \leqs \delta \sqrt{2}$ and $\lambda \geqs L$ 
\begin{equation}\label{eq:lowerboundfas}
\left| \xi_j -  \lambda^{-t m} \partial_j \left( \fy( \lambda^t y) \right)  \right| 
\geqs |\xi_j - \partial_j \fy_m (x) | - \left( | \partial_j ( \fy_m (y) - \fy_m (x) ) | + \lambda^{t(1-m)} | \partial_j p (\lambda^t y) | \right)
\geqs \ep.
\end{equation}

Let $\psi \in \Sigma_t^s (\rr d) \setminus 0$ have $\supp \psi \subseteq \rB_\delta$. 
We denote by $y' \in \rr {d-1}$ the vector $y \in \rr d$ except coordinate $j$. 
The stationary phase theorem \cite[Theorem~7.7.1]{Hormander0}
gives, for any $k \in \no$, any $h > 0$, and any $\lambda \geqs L$, if $(x,\xi) \in U$, using \eqref{eq:lowerboundfas} 
and \eqref{eq:expestimate0},
\begin{equation*}
\begin{aligned}
& |V_\psi u ( \lambda^t x, \lambda^{t(m-1)} \xi)| \\
& = (2 \pi)^{-\frac{d}{2}} \left| \int_{\rr d} e^{i ( \fy(y)  - \lambda^{t(m-1)} \la y,  \xi \ra )} \overline{\psi( \lambda^t (\lambda^{-t} y-x) )} \, \dd y \right| \\
& = (2 \pi)^{-\frac{d}{2}} \lambda^{t d} \left| \int_{| x - y | \leqs \delta} e^{i \lambda^{t m} ( \lambda^{-t m} \fy(\lambda^t y) - \la y, \xi \ra )} \overline{\psi( \lambda^t (y-x) )} \, \dd y \right| \\
& \leqs C \lambda^{t d} \int_{ | x' - y' | \leqs \delta } \sum_{n = 0}^k \lambda^{t n} \sup_{| x_j - y_j | \leqs \delta} |(\partial_j^n\psi)( \lambda^t (y-x) )| \, |\xi_j - \lambda^{- t m} \partial_j \left( \fy( \lambda^t y) \right) |^{n - 2k} \\
& \qquad \qquad \qquad \qquad \qquad \qquad \qquad \qquad \qquad \qquad \qquad \qquad \qquad   
\times \lambda^{t m (n-2k)} \, \dd y' \\
& \leqs C \lambda^{t d} \ep^{- 2 k} \int_{ | x' - y' | \leqs \delta } \sum_{n = 0}^k \sup_{|x_j-y_j| \leqs \delta} |(\partial_j^n\psi)( \lambda^t (y-x) )| \, \dd y' 
\lambda^{- t k (m-1)} \lambda^{t(1+m) (n-k)} \\
& \leqs C_h \lambda^{t d} \ep^{- 2 k} \lambda^{- s k} \sum_{n = 0}^k   h^n n!^s  \\
& = C_h \lambda^{t d} \ep^{- 2 k} \lambda^{- s k} h^k \sum_{n = 0}^k   
h^{-(k-n)} n!^s \\
& \leqs C_h \lambda^{t d} \ep^{- 2 k} \lambda^{- s k} h^k e^{s h^{-\frac1s}} \sum_{n = 0}^k  (n! (k-n)!)^s \\
& \leqs C_{s,h}  \lambda^{t d} \ep^{- 2 k} \lambda^{- s k} (2 h)^k  k!^s . 
\end{aligned}
\end{equation*}
Since $h > 0$ is arbitrary
we obtain
\begin{equation}\label{eq:STFTestimate0}
\lambda^{s k} \ep^{2k} | V_\psi u (\lambda^t x, \lambda^s \xi) | 
\leqs C_h \lambda^{t d} h^k k!^s, \quad (x,\xi) \in U, 
\end{equation}
for all $h > 0$, all $\lambda \geqs L$ and all $k \in \no$. 
Appealing to \cite[Lemma~4.1]{Rodino3} we may conclude that 
that $(x_0,\xi_0) \notin \WF^{t,t(m-1)} (u)$ and the inclusion \eqref{eq:chirpconclusion1} follows. 

Next let $d = 1$. 
If $\fy$ is even then $u$ is even, and $W = -W$ since $m$ is even, so by \eqref{eq:evensteven1} we have either $\WF^{t,t(m-1)} (u) =\emptyset$ or $\WF^{t,t(m-1)} (u) = W$. 
The former is not true since $u \notin \Sigma_t^s(\rr d)$. Thus we have proved \eqref{eq:chirpconclusion2} when $\fy$ is even. 

If $\fy$ is odd then $m$ is odd and $\check u (x) = \overline{ u(x) } = e^{- i  \fy(x)}$. 
Again $\WF^{t,t(m-1)} (u) =\emptyset$ cannot hold since $u \notin \Sigma_t^s(\rr d)$. 
If we assume that the inclusion \eqref{eq:chirpconclusion1} is strict we  get a contradiction from 
\eqref{eq:evensteven0} and \eqref{eq:evensteven2}. 
Indeed suppose e.g. 
\begin{equation*}
\WF^{t,t(m-1)} (u) 
= \{ (x, \fy_m'(x) ) \in \rr 2: \ x > 0 \}. 
\end{equation*}
By \eqref{eq:evensteven0} and \eqref{eq:evensteven2} we then get the contradiction 
\begin{align*}
\WF^{t,t(m-1)} ( \check u) 
& = \{ (x, - \fy_m'(x) ) \in \rr 2: \ x < 0 \} \\
& = \{ (x, - \fy_m'(x) ) \in \rr 2: \ x > 0 \}
= \WF^{t,t(m-1)} ( \overline{u} ).   
\end{align*}
This proves \eqref{eq:chirpconclusion2} when $\fy$ is odd. 
\end{proof}

We would also like to determine $\WF^{t,s} (u)$ when $s \neq t(m-1)$. 
The following two results treat this question. 

\begin{prop}\label{prop:chirpnegative1}
If $m \geqs 2$, $\fy$ is a real polynomial defined by 
\eqref{eq:phasefunction}, \eqref{eq:polynomial1}, \eqref{eq:principalpart},
$u$ is defined by \eqref{eq:chirpdef}, 
and $s > t(m-1) > 1$
then 
\begin{equation}\label{eq:chirpconclusion3}
\WF^{t,s} (u) \subseteq (\rr d \setminus 0) \times \{ 0 \}. 
\end{equation}
If $d = 1$ and $\fy$ is even or odd then 
\begin{equation}\label{eq:chirpconclusion4}
\WF^{t,s} (u) 
= (\ro \setminus 0) \times \{ 0 \}.
\end{equation}
\end{prop}

\begin{proof}
Suppose $(x_0, \xi_0) \in T^* \rr d$ and $\xi_0 \neq 0$, that is $\xi_{0,j} \neq 0$ for some $1 \leqs j \leqs d$. 
From \eqref{eq:phasederivative1} we obtain
\begin{equation*}
\lambda^{-t-s} \partial_j \left( \fy( \lambda^t y) \right)
= \lambda^{t (m-1)-s} \left(  \partial_j \fy_m (y) + \lambda^{t(1-m)} \partial_j p(\lambda^t y) \right).
\end{equation*}
Thus from $s > t(m-1)$, using \eqref{eq:polderivative1},
it follows that there exist $U \subseteq \rr {2d}$ such that $(x_0, \xi_0) \in U$, 
and $0 < \ep \leqs 1$, 
$L \geqs 1$ such that 
\begin{equation*}
| \xi_j - \lambda^{-t-s} \partial_j \left( \fy( \lambda^t y) \right)| \geqs \ep
\end{equation*}
when $(x,\xi) \in U$, $| x - y| \leqs \sqrt 2$ and $\lambda \geqs L$. 

Let $\psi \in \Sigma_t^s(\rr d) \setminus 0$ be such that $\supp \psi \subseteq \rB_1$. 
The stationary phase theorem \cite[Theorem~7.7.1]{Hormander0}
yields, for any $k \in \no$, any $h > 0$, and any $\lambda \geqs L$, if $(x,\xi) \in U$, 
\begin{equation*}
\begin{aligned}
|V_\psi u ( \lambda^t x, \lambda^s \xi)|
& = (2 \pi)^{-\frac{d}{2}} \left| \int_{\rr d} e^{i ( \fy(y)  - \lambda^s \la y, \xi \ra )} \overline{\psi( \lambda^t (\lambda^{-t} y-x) )} \, \dd y \right| \\
& = (2 \pi)^{-\frac{d}{2}} \lambda^{t d} \left| \int_{\rr d} e^{i \lambda^{t+s} ( \lambda^{-t-s} \fy(\lambda^t y) - \la y, \xi \ra )} \overline{\psi( \lambda^t (y-x) )} \, \dd y \right| \\
& \leqs C \lambda^{t d} \int_{ | x' - y' | \leqs 1} \sum_{n = 0}^k \lambda^{t n} \sup_{|x_j-y_j| \leqs 1} |(\partial_j^n\psi)( \lambda^t (y-x) )| \, |\xi_j - \lambda^{-t-s} \partial_j \left( \fy( \lambda^t y) \right) |^{n - 2k} \\
& \qquad \qquad \qquad \qquad \qquad \times \lambda^{(t+s) (n-2k)} \, \dd y' \\
& \leqs C_h \lambda^{t d} \ep^{- 2 k} \lambda^{- s k} \sum_{n = 0}^k   h^n n!^s \lambda^{(s+2t)(n-k)} \\
& \leqs C_h \lambda^{t d} \ep^{- 2 k} \lambda^{- s k} \sum_{n = 0}^k   h^n n!^s \\
& \leqs C_{s,h}  \lambda^{t d} \ep^{- 2 k} \lambda^{- s k} (2 h)^k  k!^s . 
\end{aligned}
\end{equation*}
Again using \cite[Lemma~4.1]{Rodino3} we may conclude that 
that $(x_0,\xi_0) \notin \WF^{t,s} (u)$ and the inclusion \eqref{eq:chirpconclusion3} follows. 

When $d = 1$ and $\fy$ is either even or odd then \eqref{eq:chirpconclusion4} follows as in the proof of Theorem \ref{thm:chirpWFGS}. 
\end{proof}

In our final result on the $t,s$-Gelfand--Shilov wave front set of a chirp function we strengthen the assumption on the polynomial $\fy_m$ to be elliptic. 

\begin{prop}\label{prop:chirpnegative2}
Let $m \geqs 2$, 
let $\fy$ be a real polynomial defined by \eqref{eq:phasefunction}, \eqref{eq:polynomial1}, \eqref{eq:principalpart}, 
suppose $\fy_m (x) \neq 0$ for all $x \in \rr d \setminus 0$, 
and let $u$ be defined by \eqref{eq:chirpdef}.
If $t(m-1) > s > 1$ then
\begin{equation}\label{eq:chirpconclusion5}
\WF^{t,s} (u) \subseteq  \{ 0 \} \times (\rr d \setminus 0). 
\end{equation}
If $d = 1$ and $\fy$ is even then 
\begin{equation}\label{eq:chirpconclusion6}
\WF^{t,s} (u) 
= \{ 0 \} \times (\ro \setminus 0).
\end{equation}
\end{prop}

\begin{proof}
Suppose $(x_0, \xi_0) \in T^* \rr d$ and $x_0 \neq 0$. 
The assumption $\fy_m (x) \neq 0$ for all $x \in \rr d \setminus 0$ and Euler's homogeneous function theorem 
imply that $\nabla \fy_m (x_0) \neq 0$, that is $\partial_j \fy_m (x_0) \neq 0$ for some $1 \leqs j \leqs d$. 
From \eqref{eq:phasederivative1} and \eqref{eq:polderivative1}
and $t(m-1) > s > 1$
it follows that there exist $U \subseteq \rr {2d}$ such that $(x_0, \xi_0) \in U$, 
$1 \leqs j \leqs d$
and $0 < \ep \leqs 1$, 
$L \geqs 1$ such that 
\begin{equation*}
| \lambda^{t + s - t m} \xi_j - \lambda^{-t m} \partial_j \left( \fy( \lambda^t y) \right)| \geqs \ep
\end{equation*}
when $(x,\xi) \in U$, $| x - y| \leqs \ep \sqrt 2$ and $\lambda \geqs L$. 

Let $\psi \in \Sigma_t^s(\rr d) \setminus 0$ be such that $\supp \psi \subseteq \rB_\ep$. 
Again by the stationary phase theorem \cite[Theorem~7.7.1]{Hormander0}
we obtain, for any $k \in \no$, any $h > 0$, and any $\lambda \geqs L$, if $(x,\xi) \in U$, 
\begin{equation*}
\begin{aligned}
|V_\psi u ( \lambda^t x, \lambda^s \xi)|
& = (2 \pi)^{-\frac{d}{2}} \left| \int_{\rr d} e^{i ( \fy(y)  - \lambda^s \la y, \xi \ra )} \overline{\psi( \lambda^t (\lambda^{-t} y-x) )} \, \dd y \right| \\
& = (2 \pi)^{-\frac{d}{2}} \lambda^{t d} \left| \int_{\rr d} e^{i \lambda^{t m} ( \lambda^{-t m} \fy(\lambda^t y) - \lambda^{t(1-m)+s} \la y, \xi \ra )} \overline{\psi( \lambda^t (y-x) )} \, \dd y \right| \\
& \leqs C \lambda^{t d} \int_{ | x' - y' | \leqs \ep}  \sum_{n = 0}^k \lambda^{t n} \sup_{|x_j-y_j| \leqs \ep} |(\partial_j^n\psi)( \lambda^t (y-x) )| \\
& \qquad \qquad \qquad \qquad \qquad \times  | \lambda^{t(1-m) + s} \xi_j - \lambda^{- t m} \partial_j \left( \fy( \lambda^t y) \right) |^{n - 2k}
\lambda^{t m (n-2k)} \, \dd y' \\
& \leqs C_h \lambda^{t d} \ep^{- 2 k} \lambda^{- s k} \sum_{n = 0}^k   h^n n!^s \lambda^{s k + t (n - m k)} \\
& \leqs C_h \lambda^{t d} \ep^{- 2 k} \lambda^{- s k} \sum_{n = 0}^k   h^n n!^s \lambda^{t ( (m-1)k  + n - m k)}\\
& \leqs C_{s,h}  \lambda^{t d} \ep^{- 2 k} \lambda^{- s k} (2 h)^k  k!^s . 
\end{aligned}
\end{equation*}
As before this shows that $(x_0,\xi_0) \notin \WF^{t,s}(u)$ and \eqref{eq:chirpconclusion3} follows. 

When $d = 1$ and $\fy$ is even then \eqref{eq:chirpconclusion6} follows as in the proof of Theorem \ref{thm:chirpWFGS}. 
\end{proof}

\section{Propagation of the $t,s$-Gelfand--Shilov wave front set for a particular evolution equation}\label{sec:schrodinger}

In \cite[Remark~4.7]{Rodino3} we discuss the initial value Cauchy problem
for the evolution equation 
in dimension $d = 1$ 
\begin{equation*}
\left\{
\begin{array}{rl}
\partial_t u(t,x) + i D_x^{m} u (t,x) & = 0, \quad m \in \no \setminus 0, \quad x \in \ro, \quad t \in \ro, \\
u(0,\cdot) & = u_0.  
\end{array}
\right.
\end{equation*}
It is a generalization of the Schr\"odinger equation for the free particle where $m = 2$. 

Here we generalize this equation into
\begin{equation}\label{eq:schrodeq3}
\left\{
\begin{array}{rl}
\partial_t u(t,x) + i p(D_x) u (t,x) & = 0, \quad x \in \rr d, \quad t \in \ro, \\
u(0,\cdot) & = u_0
\end{array}
\right.
\end{equation}
where $p: \rr d \to \ro$ is a polynomial with real coefficients of order $m \geqs 2$, that is
\begin{equation}\label{eq:polynomial2}
p (\xi) = \sum_{|\alpha| \leqs m} c_\alpha \xi^\alpha, \quad c_\alpha \in \ro.  
\end{equation}
The principal part is 
\begin{equation}\label{eq:principalpart2}
p_m (\xi) = \sum_{|\alpha| = m} c_\alpha \xi^\alpha  
\end{equation}
and there exists $\alpha \in \nn d$ such that $|\alpha| = m$ and $c_\alpha \neq 0$. 

The Hamiltonian is $p(\xi)$, and the Hamiltonian flow of the principal part $p_m(\xi)$ is given by
\begin{equation}\label{eq:hamiltonflow}
( x (t),\xi(t) ) = \chi_t (x_0, \xi_0)
= (x_0 + t \nabla p_m (\xi_0), \xi_0), \quad t \in \ro, \quad (x_0, \xi_0) \in T^* \rr d \setminus 0.
\end{equation}

The explicit solution to \eqref{eq:schrodeq3} is 
\begin{equation*}
u (t,x) 
= e^{- i t p(D_x)} u_0 
= (2 \pi)^{- \frac{d}{2}} \int_{\rr d} e^{i \la x, \xi \ra - i t p(\xi)} \wh u_0 (\xi) \dd \xi
\end{equation*}
for $u_0 \in \cS(\rr d)$. 
Thus $u (t,x) = \cK_t u_0(x)$ where $\cK_t$ is the operator with Schwartz kernel
\begin{align*}
K_t (x,y) & = (2 \pi)^{-d} \int_{\rr d} e^{i \la x-y, \xi \ra - i t p(\xi)} \dd \xi \\
& =  (2 \pi)^{- \frac{d}{2}} \cF^{-1} (e^{- i t p}) (x-y)
\end{align*}
which may be considered an element in $\cS'(\rr {2d})$. 
Thus $\cK_t$ is a convolution operator with convolution kernel
\begin{equation}\label{eq:convolutionkernel}
k_t  =  (2 \pi)^{- \frac{d}{2}} \cF^{-1} (e^{- i t p}) \in \cS'(\rr d)
\end{equation}
and we may write
\begin{equation}\label{eq:schwartzkernel1}
K_t (x,y) 
=  \left( 1 \otimes k_t \right) \circ \kappa^{-1}(x,y)
\end{equation}
where $\kappa \in \rr {2d \times {2d}}$ is the matrix defined by $\kappa(x,y) = (x+\frac{y}{2}, x - \frac{y}{2})$ for $x,y \in \rr d$. 

Since $K_t \in \cS'(\rr {2d}) \subseteq (\Sigma_r^s)' (\rr {2d})$ if $r + s > 1$, the operator $\cK_t: \Sigma_r^s (\rr d) \to (\Sigma_r^s)' (\rr d)$ is continuous for all $t \in \ro$. 

The next result shows that $\cK_t$ acts continuously on $\Sigma_r^s (\rr d)$ if $r \geqs s(m-1) > 1$, 
and the $(s(m-1),s)$-Gelfand--Shilov wave front set of the solution propagates along the Hamiltonian flow of $p_m$, 
whereas the $(r,s)$--Gelfand--Shilov wave front set is invariant when $r > s(m-1)$.

\begin{thm}\label{thm:WFgskernel}
Suppose $m \geqs 2$ and let $p$ be defined 
by \eqref{eq:polynomial2}, \eqref{eq:principalpart2}, 
and denote by \eqref{eq:hamiltonflow} the Hamiltonian flow
of the principal part $p_m$. 
Let $r, s > 0$ satisfy $r \geqs s(m-1) > 1$,
suppose $\cK_t: \cS (\rr d) \to \cS' (\rr d)$ is the continuous linear operator with Schwartz kernel 
\eqref{eq:schwartzkernel1} where $k_t$ is defined by \eqref{eq:convolutionkernel}. 
Then $\cK_t: \Sigma_{r}^s (\rr d) \to \Sigma_{r}^s (\rr d)$ is continuous, extends uniquely to a 
continuous operator $\cK_t: (\Sigma_{r}^s)' (\rr d) \to (\Sigma_{r}^s)' (\rr d)$, 
and is invertible with inverse $\cK_{t}^{-1} = \cK_{-t}$.
For $t \in \ro$ we have
\begin{align}
& \WF^{s(m-1),s}  ( \cK_t u) = \chi_t  \left( \WF^{s(m-1),s}  (u) \right), \quad u \in (\Sigma_{s(m-1)}^s)' (\rr d), \label{eq:propagation1} \\
& \WF^{r,s}  ( \cK_t u) = \WF^{r,s}  (u), \quad u \in (\Sigma_{r}^s)' (\rr d), \quad r > s(m-1). \label{eq:propagation2} 
\end{align}
\end{thm}

\begin{proof}
First we let $r = s (m-1) > 1$. 
By Theorem \ref{thm:chirpWFGS} we have 
\begin{equation*}
\WF^{s,r} ( e^{- i t p} ) 
\subseteq \{ (x, - t \nabla p_m(x) ) \in T^* \rr d: \ x \neq 0 \}
\end{equation*}
and from \eqref{eq:evensteven0} and \cite[Proposition~3.6 (i)]{Rodino3} we obtain
\begin{align*}
\WF^{r,s} ( k_t ) 
& = \WF^{r,s} ( \cF^{-1} e^{- i t p} ) 
= - \WF^{r,s} ( \cF e^{- i t p} ) \\
& = - \J \WF^{s,r} ( e^{- i t p} ) \\
& \subseteq \{ ( t \nabla p_m(x)  , x) \in T^* \rr d: \ x \neq 0 \}. 
\end{align*}

Now  \eqref{eq:schwartzkernel1}, \cite[Proposition~3.6 (ii)]{Rodino3}, 
Proposition \ref{prop:tensorWFs} and \cite[Proposition~7.1 (iii)]{Rodino3} yield
\begin{align*}
& \WF^{r,s} (K_t)
= \WF^{r,s} ( \left( 1 \otimes k_t \right) \circ \kappa^{-1} ) \\
& = 
\left( 
\begin{array}{cc}
  \kappa & 0 \\
  0 & \kappa^{-T}
  \end{array}
\right) 
\WF^{r,s} \left( 1 \otimes k_t \right) \\
& \subseteq \{ ( \kappa( x_1, x_2), \kappa^{-T}(\xi_1, \xi_2) ) \in T^* \rr {2d}: \\
& \qquad \qquad \qquad \qquad (x_1, \xi_1) \in \WF^{r,s} (1) \cup \{ 0 \},  
\ (x_2, \xi_2) \in \WF^{r,s} (k_t) \cup \{ 0 \} \} \setminus 0 \\
& \subseteq \{ ( \kappa( x_1, t \nabla p_m(x_2) ), \kappa^{-T}(0, x_2) \in T^* \rr {2d}: \ x_1, x_2 \in \rr d \} \setminus 0 \\
& = \left\{ \left( x_1 + t \frac12 \nabla p_m(x_2), x_1 - t \frac12 \nabla p_m(x_2), x_2, - x_2 \right) \in T^* \rr {2d}: \ x_1, x_2 \in \rr d  \right\} \setminus 0 \\
& = \left\{ \left( x_1 + t \nabla p_m(x_2), x_1, x_2, - x_2 \right) \in T^* \rr {2d}: \ x_1, x_2 \in \rr d \right\} \setminus 0. 
\end{align*}

Since $m \geqs 2$ we have $\nabla p_m(0) = 0$ and 
$\WF_{1}^{r,s}(K_t) = \WF_{2}^{r,s}(K_t) = \emptyset$ follows. 
Thus we may apply Theorem \ref{thm:WFGSpropagation}.  
This gives the continuity statements on $\Sigma_{r}^s (\rr d)$ and on $(\Sigma_{r}^s)' (\rr d)$. 
It also follows that $\cK_t$ is invertible with inverse $\cK_{t}^{-1} = \cK_{-t}$
on $\Sigma_{r}^s (\rr d)$ as well as on $(\Sigma_{r}^s)' (\rr d)$. 
Moreover Theorem \ref{thm:WFGSpropagation} gives for $u \in \left( \Sigma_r^s \right)' (\rr d)$
\begin{align*}
\WF^{r,s}( \cK_t u) 
& \subseteq \WF^{r,s} (K_t)' \circ \WF^{r,s} (u) \\
& = \{ (x,\xi) \in T^* \rr d: \ \exists (y,\eta) \in \WF^{r,s} (u), \ (x,y, \xi, - \eta) \in \WF^{r,s} (K_t) \} \\
& \subseteq \{ ( x_1 + t \nabla p_m(x_2) , x_2): \ (x_1,x_2) \in \WF^{r,s} (u) \} \\
& = \chi_t  \left( \WF^{r,s} (u) \right).
\end{align*}

The opposite inclusion follows from $\cK_{t}^{-1} = \cK_{-t}$, 
\begin{equation*}
\WF^{r,s} ( u) 
= \WF^{r,s} ( \cK_{-t}  \cK_t u) \\
\subseteq \chi_{-t}  \left( \WF^{r,s} ( \cK_t u) \right)
\end{equation*}
and $\chi_{-t} = \chi_t^{-1}$. 
We have proved the result when $r = s(m-1)$ and \eqref{eq:propagation1}.

It remains to consider the case $r > s(m-1) > 1$. 
By Proposition \ref{prop:chirpnegative1} we have 
\begin{equation*}
\WF^{s,r} ( e^{- i t p} ) 
\subseteq ( \rr d \setminus 0 ) \times \{ 0 \}
\end{equation*}
and from \eqref{eq:evensteven0} and \cite[Proposition~3.6 (i)]{Rodino3} we obtain
\begin{equation*}
\WF^{r,s} ( k_t ) 
= - \J \WF^{s,r} ( e^{- i t p} ) \\
\subseteq \{ 0 \} \times ( \rr d \setminus 0 ).  
\end{equation*}

Again \eqref{eq:schwartzkernel1}, \cite[Proposition~3.6 (ii)]{Rodino3}, 
Proposition \ref{prop:tensorWFs} and \cite[Proposition~7.1 (iii)]{Rodino3} yield
\begin{align*}
\WF^{r,s} (K_t)
& \subseteq \{ ( \kappa( x_1, x_2), \kappa^{-T}(\xi_1, \xi_2) ) \in T^* \rr {2d}: \\
& \qquad \qquad \qquad \qquad (x_1, \xi_1) \in \WF^{r,s} (1) \cup \{ 0 \},  
\ (x_2, \xi_2) \in \WF^{r,s} (k_t) \cup \{ 0 \} \} \setminus 0 \\
& \subseteq \{ ( \kappa( x_1, 0 ), \kappa^{-T}(0, x_2) \in T^* \rr {2d}: \ x_1, x_2 \in \rr d \} \setminus 0 \\
& = \left\{ \left( x_1 , x_1, x_2, - x_2 \right) \in T^* \rr {2d}: \ x_1, x_2 \in \rr d \right\} \setminus 0. 
\end{align*}

Again we have $\WF_{1}^{r,s}(K_t) = \WF_{2}^{r,s}(K_t) = \emptyset$, 
and Theorem \ref{thm:WFGSpropagation} 
gives the continuity statements on $\Sigma_{r}^s (\rr d)$ and on $(\Sigma_{r}^s)' (\rr d)$. 
Again $\cK_t$ is invertible with inverse $\cK_{t}^{-1} = \cK_{-t}$
on $\Sigma_{r}^s (\rr d)$ as well as on $(\Sigma_{r}^s)' (\rr d)$. 
Now Theorem \ref{thm:WFGSpropagation} gives $u \in \left( \Sigma_r^s \right)' (\rr d)$
\begin{equation*}
\WF^{r,s}( \cK_t u) 
\subseteq \WF^{r,s} (K_t)' \circ \WF^{r,s} (u) 
\subseteq \WF^{r,s} (u).
\end{equation*}

The opposite inclusion again follows from $\cK_{t}^{-1} = \cK_{-t}$ and $\chi_{-t} = \chi_t^{-1}$. 
We have proved the result when $r > s(m-1) > 1$ and \eqref{eq:propagation2}.
\end{proof}

\begin{rem}\label{rem:continuityWFK1}
In the proof of Theorem \ref{thm:WFgskernel}
the continuity of $\cK_t: \Sigma_{r}^s (\rr d) \to \Sigma_{r}^s (\rr d)$
when $r \geqs s(m-1) > 1$
is proved by means of the observation 
$\WF_{1}^{r,s}(K_t) = \WF_{2}^{r,s}(K_t) = \emptyset$
and Proposition \ref{prop:opSTFTformula}. 
It seems much more complicated to try to show this using seminorms 
on $\cK_t u$ for $u \in \Sigma_{r}^s (\rr d)$. 
\end{rem}

\begin{rem}\label{rem:continuityWFK2}
If $s(m-1) > r > 1$ and $p_m (x) \neq 0$ for all $x \in \rr d \setminus 0$
then Proposition \ref{prop:chirpnegative2} gives
\begin{equation*}
\WF^{s,r} ( e^{- i t p} ) 
\subseteq \{ 0 \} \times ( \rr d \setminus 0 ) 
\end{equation*}
so \eqref{eq:evensteven0} and \cite[Proposition~3.6 (i)]{Rodino3} give
\begin{equation*}
\WF^{r,s} ( k_t ) 
= - \J \WF^{s,r} ( e^{- i t p} ) \\
\subseteq ( \rr d \setminus 0 ) \times \{ 0 \}.  
\end{equation*}

Again \eqref{eq:schwartzkernel1}, \cite[Proposition~3.6 (ii)]{Rodino3}, 
Proposition \ref{prop:tensorWFs} and \cite[Proposition~7.1 (iii)]{Rodino3} yield
\begin{align*}
& \WF^{r,s} (K_t)
\subseteq \{ ( \kappa( x_1, x_2), \kappa^{-T}(\xi_1, \xi_2) ) \in T^* \rr {2d}: \\
& \qquad \qquad \qquad \qquad (x_1, \xi_1) \in \WF^{r,s} (1) \cup \{ 0 \},  
\ (x_2, \xi_2) \in \WF^{r,s} (k_t) \cup \{ 0 \} \} \setminus 0 \\
& \subseteq \{ ( \kappa( x_1, x_2 ), \kappa^{-T}(0,0) \in T^* \rr {2d}: \ x_1, x_2 \in \rr d \} \setminus 0 \\
& = ( \rr {2d} \setminus 0 ) \times \{ 0 \}. 
\end{align*}
In this case we cannot conclude that $\WF_{1}^{r,s}(K_t)$ and $\WF_{2}^{r,s}(K_t)$ are empty.

Thus we cannot conclude any continuity statement from Theorem \ref{thm:WFGSpropagation}. 
It is an open problem to prove or disprove continuity of $\cK_t$ on $\Sigma_{r}^s (\rr d)$
when $s(m-1) > r > 1$. 
Likewise continuity on $\Sigma_{r}^s (\rr d)$ is not known when $r + s > 1$ and $r \leqs 1$, 
nor when $r + s > 1$ and $s(m-1) \leqs 1$. 
\end{rem}

\section*{Acknowledgment}
Work partially supported by the MIUR project ``Dipartimenti di Eccellenza 2018-2022'' (CUP E11G18000350001).


\end{document}